\documentclass[
reqno]{amsart}
\usepackage{latexsym,amsmath,amssymb,amscd}
\usepackage[all]{xy}
\usepackage{amsthm}
\usepackage{enumerate}
\usepackage{hyperref}

\makeatletter
  \newcommand\@dotsep{4.5}
  \def\@tocline#1#2#3#4#5#6#7{\relax
     \ifnum #1>\c@tocdepth 
     \else
     \par \addpenalty\@secpenalty\addvspace{#2}%
     \begingroup \hyphenpenalty\@M
     \@ifempty{#4}{%
     \@tempdima\csname r@tocindent\number#1\endcsname\relax
        }{%
         \@tempdima#4\relax
           }%
      \parindent\z@ \leftskip#3\relax \advance\leftskip\@tempdima\relax
      \rightskip\@pnumwidth plus1em \parfillskip-\@pnumwidth
       #5\leavevmode\hskip-\@tempdima #6\relax
       \leaders\hbox{$\m@th
       \mkern \@dotsep mu\hbox{.}\mkern \@dotsep mu$}\hfill
       \hbox to\@pnumwidth{\@tocpagenum{#7}}\par
       \nobreak
        \endgroup
         \fi}
\makeatother 

\DeclareMathAlphabet\mathbfcal{OMS}{cmsy}{b}{n}

\newtheorem{theorem}{Theorem}

\newtheorem{lemma}[theorem]{Lemma}
\newtheorem{notation}[theorem]{Notation}

\newtheorem{proposition}[theorem]{Proposition}
\newtheorem{remark}[theorem]{Remark}

\theoremstyle{definition}
\newtheorem{definition}[theorem]{Definition}
\newtheorem{example}[theorem]{Example}

\numberwithin{theorem}{section}
\numberwithin{equation}{section}

\newcommand{\CC}{{\mathbb C}}

\newcommand{\KK}{{\mathbb K}}
\newcommand{\RR}{{\mathbb R}}

\newcommand{\NN}{{\mathbb N}}

\newcommand{\Bc}{{\mathcal B}}

\newcommand{\Dc}{{\mathcal D}}

\newcommand{\Jc}{{\mathcal J}}

\newcommand{\Lc}{{\mathcal L}}

\newcommand{\Oc}{{\mathcal O}}

\newcommand{\Sc}{{\mathcal S}}

\newcommand{\Vc}{{\mathcal V}}

\newcommand{\Wc}{{\mathcal W}}

\newcommand{\Pg}{{\mathfrak P}}

\newcommand{\ag}{{\mathfrak a}}

\renewcommand{\gg}{{\mathfrak g}}
\newcommand{\hg}{{\mathfrak h}}

\newcommand{\n}{{\mathfrak n}}

\newcommand{\pg}{{\mathfrak p}}

\newcommand{\zg}{{\mathfrak z}}

\renewcommand{\1}{{\bf 1}}

\newcommand{\ad}{{\rm ad}}

\newcommand{\card}{{\rm card}\,}

\newcommand{\de}{{\rm d}}

\newcommand{\id}{{\rm id}}
\newcommand{\ie}{{\rm i}}

\newcommand{\jump}{{\rm jump}\,}

\newcommand{\Ker}{{\rm Ker}\,}

\newcommand{\Ran}{{\rm Ran}\,}

\newcommand{\spa}{{\rm span}\,}

\newcommand{\Gr}{{\rm Gr}}
\newcommand{\alg}{{\rm alg}}
\newcommand{\rank}{{\rm rank}\,}

\newcommand{\Sbd}{{\mathfrak S}}

\begin{document}

\title[Continuous selection of Lagrangian subspaces]{Continuous selection of Lagrangian subspaces}
\author[Ingrid Belti\c t\u a]{Ingrid Belti\c t\u a}  
\author[Daniel Belti\c t\u a]{Daniel Belti\c t\u a}
\address{Institute of Mathematics ``Simion Stoilow'' 
of the Romanian Academy,   
P.O. Box 1-764, Bucharest, Romania}
\email{ingrid.beltita@gmail.com, Ingrid.Beltita@imar.ro}

\address{Institute of Mathematics ``Simion Stoilow'' 
of the Romanian Academy,   
P.O. Box 1-764, Bucharest, Romania}
\email{beltita@gmail.com, Daniel.Beltita@imar.ro}

\keywords{presymplectic form, isotropic subspace, completely solvable Lie algebra}
\subjclass[2020]{Primary 17B30; Secondary 15A69, 37J37, 53D12}

\begin{abstract}
We study continuous selections of the set-valued map that takes every skew-symmetric bilinear form on a vector space to its corresponding set of maximal isotropic subspaces. 
Applications are made to establishing continuity properties of the Vergne polarizing subalgebras of completely solvable Lie algebras in terms of Schubert cells of suitable Grassmann manifolds.
\end{abstract}

\maketitle


\section{Introduction}

Polarizing subalgebras of Lie algebras and, more generally, Lagrangian subspaces of presymplectic vector spaces, provide a nice illustration of the applications of linear algebra to areas such as symplectic geometry, geometric quantization, and in particular explicit constructions of unitary representations associated to coadjoint orbits of Lie groups. 
These explicit constructions are important not only in representation theory \cite{CG90} but also for integrable systems \cite{BGR17a}, \cite{BGR17b}, some topics in linear partial differential equations \cite{HN85} or construction of frames associated to Lie group representations \cite{Ou19}. 

In the study of topological properties of the unitary dual spaces of Lie groups (cf.,  e.g., \cite{BB21a}, \cite{BB21b}, \cite{BBL17}, \cite{LMB11}), 
the mere existence of polarizations is, however, not enough, 
and one would actually need continuous selections of polarizations, in a sense that will be specified shortly. 
Therefore, in the present paper we study continuous selections of Lagrangian subspaces, with emphasis on the so-called Vergne polarizations, which were first constructed in \cite{Ve70} and played an important role in representation theory  and related areas ever since. 
See for instance \cite{CG90} for the role of Vergne polarizations in representation theory of nilpotent Lie groups and \cite{Ou15} for computational aspects of these polarizations.

In order to describe what we mean here by continuous selections, 
we first recall that 
a presymplectic structure on a finite-dimensional real vector space $\Vc$ is just a skew-symmetric bilinear form $B\colon\Vc\times\Vc\to\RR$. 
The isotropic subspaces of $\Vc$ with respect to $B$ are the linear subspaces $\Sc\subseteq\Vc$ satisfying $B\vert_{\Sc\times\Sc}=0$, 
and every maximal isotropic subspace is called a Lagrangian subspace (or a polarization) of~$\Vc$ with respect to $B$. 
 The set of all presymplectic structures of $\Vc$ can be identified with the linear dual space  $(\wedge^2\Vc)^*$ of the second exterior power of~$\Vc$, 
while the set of all linear subspaces of $\Vc$ has the natural structure of a compact manifold, called the Grassmann manifold of~$\Vc$. 
For every presymplectic structure $B\in(\wedge^2\Vc)^*$ its corresponding set of Lagrangian subspaces, denoted
$\Pg(B)$  is a closed non-empty subset of the Grassmann manifold $\Gr(\Vc)$. 
(See also \cite{AL09} for an infinite dimensional version of $\Pg(B)$.)
In this paper we study continuous selections of the set-valued mapping 
$$ \Pg\colon (\wedge^2\Vc)^*\to 2^{\Gr(\Vc)},\quad B\mapsto \Pg(B)$$
that is, continuous mappings $\sigma\colon D_\sigma\to\Gr(\Vc)$ 
whose domain $D_\sigma\subseteq (\wedge^2\Vc)^*$ should be as large as possible, 
satisfying $\sigma(B)\in\Pg(B)$ for all $B\in D_\sigma$. 
(See Theorems \ref{Lcont} and \ref{Sopt3}.)

When $\Vc$ is the underlying vector space of a Lie algebra $\gg$, we moreover obtain continuous selections whose values are subalgebras of~$\gg$ rather than just linear subspaces (Theorem~\ref{polcont}). 
This is actually the main motivation of the present research, and the applications of our results to the study of $C^*$-algebras of solvable Lie groups will be presented in a sequel to this paper. 

The structure of this paper is as follows: 
In Section~\ref{Sect2} we discuss the semicontinuity properties of the set of Lagrangian subspaces, regarded as a set-valued mapping defined on the set of all presymplectic structures of a fixed vector space (Proposition~\ref{Luppsemi}). 
In Section~\ref{Sect3} we establish the results that we need on continuity of null-spaces of presymplectic structures (Theorem~\ref{nullcont-opt}). 
Section~\ref{Sect4} includes one of our main results on continuous selections of Lagrangian subspaces (Theorem~\ref{Lcont}). 
The relation between that result and the Schubert cells in Grassmann manifolds is explored in Section~\ref{Sect5}, 
where we need to develop a more general version of some results of \cite{Cu88} and \cite{Cu92} in order to prove Theorem~\ref{Sopt3}. 
In Section~\ref{Sect6} we apply the preceding results to establishing continuity properties of Vergne polarizations in completely solvable Lie algebras (Theorem~\ref{polcont}). 
Finally, in Section~\ref{Sect7}, we discuss in some detail the illustration of the general results by three specific examples of Lie algebras: 1. the Lie algebras with abelian hyperplane ideals (whose corresponding simply connected Lie groups are sometimes called generalized $ax+b$-groups, e.g., in \cite{BB18}); 2. the free 2-step nilpotent Lie algebra with three generators (whose Vergne polarizations were also studied  in \cite{Ou15} for arbitrarily many generators, however only on a generic subset of its linear dual); 3. the indecomposable 3-step nilpotent, 5-dimensional Lie algebra. 
In the classification of low-dimensional nilpotent Lie algebras, the last two algebras are denoted by  $\gg_{6,15}$ and $\gg_{5,4}$, respectively, 
cf. \cite{Dix58} and \cite{Ni83}.

\subsection*{General notation}
Unless otherwise mentioned, $\KK$ stands for a field of characteristic zero. 
If $\Vc$ is a vector space over $\KK$, 
we denote by $\Vc^*$ the space of all $\KK$-linear functionals $\xi\colon \Vc\to\KK$, with the corresponding duality pairing 
$\langle\cdot,\cdot\rangle\colon\Vc^*\times\Vc\to\KK$, $\langle \xi,v\rangle:=\xi(v)$. 
Similarly, we denote by $(\wedge^2\Vc)^*$ the set of all skew-symmetric bilinear forms $B\colon \Vc\times\Vc\to\KK$. 
For every $B\in(\wedge^2\Vc)^*$ and $v,w\in\Vc$ we write $v\perp_B w$ if $B(v,w)=0$. 
For every subset $S\subseteq\Vc$ we define 
$$S^{\perp_B}:=\{w\in\Vc\mid v\perp_B w\text{ for all }v\in S\}.$$ 
We then define the \emph{null-space}  mapping 
$$N\colon (\wedge^2\Vc)^*\to\Gr(\Vc),\quad N(B):=\Vc^{\perp_B}$$
where $\Gr(\Vc)$ is the Grassmann manifold of $\Vc$, that is, the set of all linear subspaces of $\Vc$. 
For every integer $k\ge0$ we also denote $\Gr_k(\Vc):=\{\Sc\in\Gr(\Vc)\mid\dim_{\KK}\Sc=k\}$ 
and we define 
$$ (\wedge^2\Vc)^*_k:=N^{-1}(\Gr_k(\Vc))=\{B\in \wedge^2\Vc\mid \dim N(B)=k\}.$$
If $n:=\dim_\KK\Vc<\infty$, then one has the disjoint union 
$$(\wedge^2\Vc)^*_k=\bigsqcup_{k=0}^n(\wedge^2\Vc)^*_k.$$
If moreover $\KK\in\{\RR,\CC\}$, then  $(\wedge^2\Vc)^*_0$ is the set of all symplectic structures of $\Vc$ and this is a dense open subset of $(\wedge^2\Vc)^*$. 

We denote by $M_n(\KK)$ the space of all square matrices $A=(a_{ij})_{1\le i,j\le n}$  with entries in $\KK$, 
and the transpose of such a matrix is $A^\top:=(a_{ji})_{1\le i,j\le n})$.

\section{Upper semicontinuity of the set of Lagrangian subspaces}
\label{Sect2}

In this section we introduce that set-valued mapping that takes every presymplectic form on a fixed vector space to its corresponding set of Lagrangian subspaces and we establish the basic semicontinuity property of that mapping that holds on its entire domain
of definition. 

Let $\Vc$ be any finite-dimensional vector space over $\KK\in\{\RR,\CC\}$ 
and we fix a closed subset $F\subseteq\Gr(\Vc)$. 
\begin{definition}\label{Lpolarizations}
	\normalfont
	For every $B\in(\wedge^2\Vc)^*$ we define 
	$$\begin{aligned}
	\Sbd_F(B)
	&:=\{\Wc\in F\mid B\vert_{\Wc\times\Wc}=0\}, \\
	\Pg_F(B)
	&:=\{\pg\in \Sbd_F(B)\mid \dim\pg =\max_{\Wc\in\Sbd_F(B)}\dim\Wc\} \\
	\end{aligned}$$
	hence $\Sbd_F(B)$ is the set of all $B$-isotropic subspaces that belong to~$F$, while $\Pg_F(B)$ is  the set of all \emph{$F$-Lagrangian subspaces} (or \emph{$F$-polarizations}) with respect to~$B$. 
	If $F=\Gr(\Vc)$, then we write simply $\Sbd(B)$ and $\Pg(B)$ instead of 
	$\Sbd_{\Gr(\Vc)}(B)$ and $\Pg_{\Gr(\Vc)}(B)$, respectively. 
\end{definition}

\begin{remark}\label{Lgrcomp}
	\normalfont
	If  $\lim\limits_{j\to\infty}\Wc_j=\Wc$ in $\Gr(\Vc)$, 
	then for every $w\in\Wc$ there exist $w_j\in\Wc_j$ for all $j\in\NN$ with 
	$\lim\limits_{j\to\infty}w_j=w$ in $\Vc$, as can be seen by using
	a local chart of the smooth manifold $\Gr(\Vc)$. 
\end{remark} 

\begin{notation}
	\normalfont
	For any set $X$ we denote by $2^X$ the set of all subsets of $X$. 
	If moreover $X$ is a topological space, then for every sequence $A_0, A_1,\dots\in 2^X$ we define 
	$$\liminf_{j\to\infty}A_j:=\Bigl\{a\in X\mid 
	\Bigl(\exists (a_j)_{j\in\NN}\in\prod\limits_{j\in\NN}A_j\Bigr)\ \lim_{j\to\infty}a_j=a\Bigr\} $$
	and 
	$$\limsup_{j\to\infty}A_j:=\bigcup_{\theta\in S_{\uparrow}(\NN)}\liminf_{j\to\infty}A_{\theta(j)},$$
	where we denote by $S_{\uparrow}(\NN)$ the set of all strictly increasing functions $\theta\colon\NN\to\NN$. 
	
	We always have $\liminf\limits_{j\to\infty}A_j\subseteq\limsup\limits_{j\to\infty}A_j$, and if these sets are equal, 
	then we denote them by $\lim\limits_{j\to\infty}A_j$. 
	Other basic properties of these notions can be found in \cite[\S 29]{Ku66} with differing notation. 
\end{notation}

\begin{remark}
	\normalfont
	If there exists $A\in 2^X$ such that $A_j=A$ for all $j\in\NN$, 
	then we have $\liminf\limits_{j\to\infty}A_j=\limsup\limits_{j\to\infty}A_j$ ($=\lim\limits_{j\to\infty}A_j$). 
	If moreover $X$ is a metrizable space, then $\lim\limits_{j\to\infty}A_j$ is equal to the closure of $A$. 
\end{remark}

\begin{definition}
	\normalfont 
	Assume that $X$ is a compact metric space. 
	For any metric space~$T$ and any function 
	$f\colon T\to 2^X$ whose values are closed subsets of $X$, 
	we say that $f$ is \emph{upper semicontinuous} if 
	whenever $\lim\limits_{j\to\infty}t_j=t$ in $T$, 
	we have $\limsup\limits_{j\to\infty}f(t_j)\subseteq f(t)$. 
	
	On the other hand, the function $f$ is \emph{lower semicontinuous} 
	if whenever $\lim\limits_{j\to\infty}t_j=t$ in $T$, 
	we have $f(t)\subseteq \liminf\limits_{j\to\infty}f(t_j)$.
\end{definition}

\begin{remark}
	\normalfont
	It follows by \cite[\S 43, II, Th. 1]{Ku68} that the above definition 
	is equivalent to the definition of upper (respectively, lower) semicontinuity in \cite[\S 18, I]{Ku66}, 
	namely that for every open set $D\subseteq X$ the set 
	$f^{-1}(D):=\{t\in T\mid f(t)\subseteq D\}$ 
	is open (respectively, closed) in~$T$.  
\end{remark}

\begin{lemma}\label{Lacc}
	If $\lim\limits_{i\to\infty}B^{(i)}=B$ in $(\wedge^2\Vc)^*$ with $N(B^{(i)})\in F$ for every $i\in\NN$, then the following asertions hold: 
	\begin{enumerate}[{\rm(i)}]
		\item\label{Lacc_item1} 
		If $\Wc$ is any cluster point of the sequence $\{N(B^{(i)})\}_{i\in\NN}$ in $\Gr(\Vc)$, 
		then $\Wc\in F$ and $\Wc\subseteq N(B)$.  
		\item\label{Lacc_item2} 
		If  
		$\Wc^{(i)}\in\Sbd_F(B^{(i)})$ is arbitrarily selected, 
		 any cluster point $\Wc$ of the sequence $\{\Wc^{(i)}\}_{j\in\NN}$ in $\Gr(\Vc)$ belongs to $\Sbd_F(B)$. 
		\item\label{Lacc_item3} 
		If for all $i\in\NN$ we have $\dim N(B^{(i)})=\dim N(B)$ and we select arbitrarily 
		$\pg^{(i)}\in\Pg_F(B^{(i)})$, 
		then  any cluster point $\pg$ of the sequence $\{\pg^{(i)}\}_{i\in\NN}$ in $\Gr(\gg)$ belongs to $\Pg_F(\xi)$.
	\end{enumerate}
\end{lemma}

\begin{proof}
	For Assertion~\eqref{Lacc_item1}, we have $N(B^{(i)})\in F$ for every $i\in\NN$, 
	hence $\Wc\in F$ since $F$ is a closed subset of $\Gr(\Vc)$. 
	To prove that $\Wc\subseteq N(B)$, let $w\in\Wc$ be arbitrarily chosen. 
	Since $\Wc$ is a cluster point of the sequence $\{N(B^{(i)})\}_{i\in\NN}$, 
	we may assume 
	$\Wc=\lim\limits_{i\to\infty}N(B^{(i)})$ in $\Gr(\gg)$, 
	by selecting a suitable subsequence. 
	Then by Remark~\ref{Lgrcomp} 
	there exist vectors $w_i\in N(B^{(i)})$ for all $i\in\NN$, with $\lim\limits_{i\to\infty}w_i=w$, 
	hence 
	$$(\forall v\in\Vc)\quad B(w,v)=\lim\limits_{i\to\infty} B^{(i)}(w_i,v)=0$$
	and this shows that $w\in N(B)$. 
	
	For Assertion~\eqref{Lacc_item2}, again by selecting a suitable subsequence, 
	we may assume $\Wc=\lim\limits_{i\to\infty}\Wc^{(i)}$ in $\Gr(\Vc)$, 
	which implies $\Wc\in F$ since $F\subseteq\Gr(\Vc)$ is closed and
	 $\Wc^{(i)}\in\Sbd_F(B^{(i)})\subseteq F$ for all $i\in\NN$. 
	Moreover for all $v,w\in\pg$ there exist $v_i,w_i\in\Wc^{(i)}$ for all $i\in\NN$ with 
	$\lim\limits_{i\to\infty}v_i=v$ and $\lim\limits_{i\to\infty}w_i=w$ 
	(again by Remark~\ref{Lgrcomp}). 
	Therefore 
	$$B(v,w)=\lim\limits_{i\to\infty}B_i(v_i,w_i)=0$$
	and thus $\Wc\in\Sbd_F(\xi)$. 
	
	For Assertion~\eqref{Lacc_item3}, we may assume again 
	$\pg=\lim\limits_{i\to\infty}\pg^{(i)}$ in $\Gr(\Vc)$.  
	If $m:=\dim\Vc$ and $m_0:=\dim N(B)$, and $k:=(m+m_0)/2$, 
	then we have $\pg^{(i)}\in\Gr_k(\Vc)$ for all $i\in\NN$, hence also $\pg\in\Gr_k(\gg)$, 
	that is, $\dim\pg$ is equal to the dimension of any $F$-polarization at $B\in(\wedge^2\Vc)^*$. 
	On the other hand $\pg\in\Sbd_F(B)$ by Assertion~\eqref{Lacc_item2}, 
	hence $\pg\in\Pg_F(\xi)$, and this concludes the proof. 
\end{proof}

\begin{proposition}\label{Luppsemi}
	For each integer $k\ge 0$, the map 
	$\Pg_F\vert_{(\wedge^2\Vc)^*_k}\colon(\wedge^2\Vc)^*_k\to 2^{\Gr_{\alg}(\gg)}$, $B\mapsto \Pg_F(B)$
	is upper semicontinuous. 
\end{proposition}

\begin{proof}
	Using Lemma~\ref{Lacc}\eqref{Lacc_item3} for constant sequences of vectors in $(\wedge^2\Vc)^*$, 
	it follows that $\Pg_F(B)$ is a closed subset of $\Gr(\gg)$ for every $B\in(\wedge^2\Vc)^*$. 
	Then, using again Lemma~\ref{Lacc}\eqref{Lacc_item3} for convergent sequences in $(\wedge^2\Vc)^*_k$, 
	we obtain the assertion. 
\end{proof}

\section{Continuity of null-spaces}
\label{Sect3}

In this section we establish a key result in constructing a continuous selection of Lagrangian subspaces, in the form that will be needed later on.

\begin{remark}[Grassmannian of a complexified vector space]\label{complexif}
	\normalfont
	Let $\Vc$ be any finite-dimensional real vector space 
	with its complexification $\Wc:=\CC\otimes_{\RR}\Vc=\Vc\dotplus\ie \Vc$. 
	Fix some integer $k$ with $1\le k\le \dim\Vc$. 
	We will denote by $\Gr_k(\Vc)$ the set of all $k$-dimensional linear subspaces of~$\Vc$, 
	and by $\Gr_k(\Wc)$ the set of all $k$-dimensional \emph{complex} linear subspaces of~$\Wc$.
	
	The canonical conjugation of $\Wc$ associated with $\Vc$, 
	$$C\colon\Wc\to\Wc,\quad C(x+\ie y)=x-\ie y,$$
	gives rise to a diffeomorphism 
	$$\alpha_C\colon \Gr_k(\Wc)\to \Gr_k(\Wc),\quad \alpha_C(Z)=C(Z)$$
	which is involutive, in the sense that $\alpha_C\circ\alpha_C=\id_{\Gr_k(\Wc)}$. 
	Regarding the Grassmann manifolds as homogeneous spaces, 
    the complexification map 
	$$\Gr_k(\Vc)\to\Gr_k(\Wc),\quad X\mapsto\CC\otimes_{\RR}X=X+\ie X $$
	is a diffeomorphism onto its image, 
	and its image is the real submanifold of $\Gr_k(\Wc)$ defined as the fixed-point set of the map $\alpha_C$. 
\end{remark}

\begin{lemma}\label{cont}
	Let $n\ge 1$ be any integer and 
	$\tau\colon T\to M_n(\RR)$ be any continuous map on some topological space $T$, 
	satisfying the following conditions:
	\begin{enumerate}[{\rm(i)}]
		\item for all $t\in T$ we have 
		$\tau(t)^\top=-\tau(t)$; 
		\item there exists an integer $k\ge 1$ with $\dim(\Ker \tau(t))=k$ for all $t\in T$.  
	\end{enumerate}
	Then the map 
	$$\widetilde{\tau}\colon T\to\Gr_k(\RR^n),\quad \widetilde{\tau}(t):=\Ker\tau(t)$$
	is continuous. 
\end{lemma}

\begin{proof}
	We regard the values of $\tau$ as skew-adjoint operators on the complex Hilbert space $\CC^n$ 
	with its canonical scalar product, and denote by $\tau_{\CC}\colon T\to\Bc(\CC^n)$ 
	the map obtained in this way, which is clearly continuous since so is $\tau$. 
	Then for every $t\in T$ and $x,y\in\RR^n$ we have 
	$(\tau_{\CC}(t))(x+\ie y)=\tau(t)x+\ie \tau(t)y$ with $\tau(t)x,\tau(t)y\in\RR^n$, 
	and this implies that 
	\begin{equation}\label{cont_proof_eq0}
	(\forall t\in T)\quad\Ker\tau_{\CC}(t)=(\Ker\tau(t))\dotplus\ie(\Ker\tau(t)),
	\end{equation} 
	hence 
	\begin{equation}\label{cont_proof_eq1}
	(\forall t\in T)\quad \dim_{\CC}(\Ker\tau_{\CC}(t))=\dim_{\RR}(\Ker \tau(t))=k.
	\end{equation}
	On the other hand, if we denote by $E_0(t)\in\Bc(\CC^n)$ the orthogonal projection 
	onto $\Ker\tau_{\CC}(t)$, then the above equalities imply that the map 
	$E_0\colon T\to\Bc(\CC^n)$ is continuous at any $t_0\in T$.  
	In fact, if we denote by $\Gamma$ any circle in $\CC$ with its center at $0$ 
	and whose exterior contains all the non-zero eigenvalues of $\tau_{\CC}(t_0)$, 
	then we have 
	$$E_0(t_0)=\frac{1}{2\pi\ie}\int\limits_\Gamma (z\1-\tau_{\CC}(t_0))^{-1}\de z$$
	and this implies that there exists a neighborhood $V$ of $t_0\in T$ 
	such that for every $t\in T$ the sum of algebraic multiplicities 
	of egenvalues of $\tau_{\CC}(t)$ contained in the interior of $\Gamma$ is equal to the rank of the projection $E_0(t_0)$, 
	and moreover the spectral projecion of $\tau_{\CC}(t)$ corresponding to the interior of $\Gamma$ 
	depends continuously on $t\in V$ 
	(see \cite[Ch. II, \S 5, Eq. (5.2)]{Ka82}). 
	But then \eqref{cont_proof_eq1} implies that for every $t\in V$ 
	the only eigenvalue of $\tau_{\CC}(t)$ which belongs to the interior of $\Gamma$ is $0$, 
	and then the map $V\to\Bc(\CC^n)$, $t\mapsto E_0(t)$, is continous. 
	Consequently the map 
	$$T \to\Gr_k(\CC^n),\quad t\mapsto\Ker\tau_{\CC}(t)$$
	is continuous (see also \cite[Th. I-2-6]{FGP94}). 
	Now, using \eqref{cont_proof_eq0} and Remark~\ref{complexif}, 
	one obtains the assertion. 
\end{proof}

\begin{proposition}\label{nullcont}
	Let $\KK\in\{\RR,\CC\}$ and assume that $\Vc$ is a $\KK$-vector space with $n:=\dim_\KK\Vc<\infty$. 
	Then the mapping 
	$(\wedge^2\Vc)^*_k\to\Gr_k(\gg)$, $B\mapsto N(B)$, 
	is continuous for $k=0,1,\dots,n$. 
\end{proposition}

\begin{proof}
	Let $\tau\colon(\wedge^2\Vc)^*\to\Lc(\Vc,\Vc^*)$, $(\tau(B))x=B(x,\cdot)$. 
	It is clear that $\langle\tau(B)v,w\rangle=-\langle\tau(B)w,v\rangle$ for all $B\in (\wedge^2\Vc)^*$ and $v,w\in\Vc$. 
	Therefore, if we select any basis in $\Vc$ and we use its dual basis in $\Vc^*$ 
	in order to write the values of $\tau$ as square matrices, then $\tau(t)$ will be given by 
	a skew-symmetric matrix for all $B\in (\wedge^2\Vc)^*$. 
	
	On the other hand, we have  that
	$N(B)=\Ker(\tau(B))$ for every $B\in (\wedge^2\Vc)^*$.
	Thus the assertion follows by Lemma~\ref{cont}. 
\end{proof}

\begin{theorem}
\label{nullcont-opt}
	Let $\Vc$ be a vector space over $\KK\in\{\RR,\CC\}$ with $n:=\dim_\KK\Vc<\infty$, and let $S\subseteq(\wedge^2\Vc)^*$ be any subset. 
Then the mapping $N\vert_S\colon S\to\Gr(\Vc)$ is continuous 
if and only if for every $k\in\{0,\dots,n\}$ the set $S\cap(\wedge^2\Vc)^*_k$ is relatively closed in~$S$. 
\end{theorem}

\begin{proof}
If  the set $S_k:=S\cap(\wedge^2\Vc)^*_k$ is relatively closed in~$S$  for every $k\in\{0,\dots,n\}$, then the disjoint union $S=S_0\sqcup S_1\sqcup\cdots\sqcup S_n$ is a partition of $S$ into relatively open subsets. 
Since the mapping $N\vert_{S_k}\colon S_k\to\Gr_k(\Vc)$ is continuous by Proposition~\ref{nullcont} for every $k\in\{0,\dots,m\}$, it then follows that the mapping $N\vert S\colon S\to\Gr(\Vc)$ is continuous. 

Conversely, let us assume that the mapping $N\vert_S\colon S\to\Gr(\Vc)$ is continuous. 
We  must prove that if  $k\in\{0,\dots,m\}$, $B\in S$, and $\{B^{(i)}\}_{i\in\NN}$ is any sequence in $S_k$ 
	with $\lim\limits_{i\to\infty}B^{(i)}=B$, then $B\in S_k$, 
	that is, $\dim N(B)=k$. 
Since  the mapping $N\vert_S\colon S\to\Gr(\Vc)$ is assumed to be continuous, we have $\lim\limits_{i\to\infty}N(B^{(i)})=N(B)$ in $\Gr(\Vc)$. 
	The linear subspaces of $\Vc$ of different dimensions belong to different connected components of $\Gr(\Vc)$ 
	and the connected components are closed subsets. 
	Since $N(B^{(i)})\in\Gr_k(\Vc)$, it then follows that $N(B)\in\Gr_k(\Vc)$, 
	that is, $B\in S_k$,  and this completes the proof. 
\end{proof}

\begin{remark}\label{Lopt1}
\normalfont
For the sake of completeness we recall that if $\Vc$ is any finite-di\-men\-sion\-al vector space over $\KK\in\{\RR,\CC\}$, 
then the mapping 
$\dim N\colon (\wedge^2\Vc)^*\to\NN$ is upper semicontinuous. 
That is, if 
  $\lim\limits_{i\to\infty}B^{(i)}=B$ in $(\wedge^2\Vc)^*$,  
 	then there exists $i_1\in\NN$ with $\dim N(B^{(i)})\le\dim N(B)$ for every $i\ge i_1$. 

In fact, let $\tau\colon\Vc\to\Lc(\Vc,\Vc^*)$, $\tau(B)v:=B(v,\cdot)$, 
	so that $N(B)=\Ker\tau(B)$, 
	as in the proof of Proposition~\ref{nullcont}. 
	Denoting $r:=\dim\Ran\tau(B)=\dim\Vc-\dim N(B)$, 
	and selecting some bases in $\Vc$ and $\Vc^*$, respectively, 
	it follows that a certain $r\times r$ minor of the matrix of $\tau(B)$ is different from zero. 
	Since $\lim\limits_{i\to\infty}B^{(i)}=B$, it follows that the corresponding minor of $\tau(B^{(i)})$ 
	is different from zero for every $i\ge i_1$, for a suitable  $i_1\in\NN$. 
	Then for every $i\ge i_1$ we have $r\le \dim\Ran\tau(B^{(i)})=\dim\Vc-\dim N(B^{(i)})$, 
	hence $\dim N(B^{(i)})\le \dim\Vc-r=\dim N(B)$, and we are done. 
\end{remark}

\section{Continuity of 
	Lagrangian subspaces} 
\label{Sect4}

In this section we construct a continuous selection of Lagrangian subspaces on suitable subsets of the set of presymplectic structures.

\begin{lemma}\label{addit}
	If $\Vc$ and $\Wc$ are finite-dimensional Hilbert spaces over $\KK\in\{\RR,\CC\}$, 
	then for all $A_1,\dots,A_m\in\Bc(\Vc,\Wc)$ we have $\Ran(A_1A_1^*+\cdots+A_mA_m^*)=\Ran A_1+\cdots+\Ran A_m$. 
\end{lemma}

\begin{proof}
	 Defining $A\colon\Vc^{\oplus m}\to\Wc$, $A(v_1,\dots,v_m):=A_1v_1+\cdots+A_mv_m$, 
	we obtain $\Ran A=\Ran A_1+\cdots+\Ran A_m$. 
	
	On the other hand, $(\Ran A)^\perp=\Ker A^*=\Ker AA^*=(\Ran AA^*)^\perp$, 
	hence $\Ran A=\Ran AA^*$. 
	And finally, it is easily checked that $A^*\colon\Wc\to\Vc^{\oplus m}$ is given by $A^*w=(A_1^*w,\dots,A_m^*w)$ for all $w\in\Wc$, hence $AA^*=A_1A_1^*+\cdots+A_mA_m^*$, 
	and then the assertion follows directly. 
\end{proof} 

\begin{lemma}\label{rancont}
	Let $\Vc$ be any finite-dimensional real vector space 
	and for any integer $k\ge 1$ define 
	$\Bc_k(\Vc):=\{T\in\Bc(\Vc)\mid \rank T=k\}$. 
	Then the map  
	$\Bc_k(\Vc)\to\Gr_k(\Vc)$, $T\mapsto \Ran T$, 
	is continuous. 
\end{lemma}

\begin{proof} After fixing a scalar product on $\Vc$, 
	it is enough to prove that the map 
	$\Bc_k(\Vc)\to \Bc_k(\Vc)$, $T\mapsto P_{\Ran T}$, 
	is continuous, and this follows from Lemma~\ref{cont}.  
	Here, for every linear subspace $\Wc\subseteq\Vc$ 
	we denote by $P_{\Wc}\in\Bc(\Vc)$ the orthogonal projection onto~$\Wc$. 
\end{proof}

\begin{proposition}\label{addcont}
	Let $\Vc$ be any finite-dimensional vector space over $\KK\in\{\RR,\CC\}$. 
	For any integers $k_1,\dots,k_m,k\ge 0$ define 
	$$\Gr_{k_1,\dots,k_m}^k(\Vc):=\{(\Vc_1,\dots,\Vc_m)\in \Gr_{k_1}(\Vc) \times \cdots\times \Gr_{k_m}(\Vc)\mid 
	\Vc_1+\cdots+\Vc_m\in \Gr_k(\Vc)\}$$
	regarded as a topological subspace of $\Gr(\Vc) \times \cdots\times \Gr(\Vc)$. 
	Then the map 
	$$\Gr_{k_1,\dots,k_m}^k(\Vc)\to \Gr_k(\Vc),\quad 
	(\Vc_1,\dots,\Vc_m)\mapsto \Vc_1+\cdots+\Vc_m$$
	is continuous. 
\end{proposition}

\begin{proof}
	We endow $\Vc$ with a structure of Hilbert space over $\KK\in\{\RR,\CC\}$. 
	It then follows by Lemma~\ref{addit} that 
	$$(\forall \Vc_1,\dots,\Vc_m\in\Gr(\Vc))\quad \Vc_1+\cdots+\Vc_m=\Ran(P_{\Vc_1}+\cdots+P_{\Vc_m}).$$
	This shows that the map referred to in the statement 
	is the composition of the maps 
	\begin{itemize} 
		\item $\Gr_{k_1,\dots,k_m}^k(\Vc)\to \Dc$, $(\Vc_1,\dots,\Vc_m)\mapsto (P_{\Vc_1},\dots,P_{\Vc_m})$, 
		\item $\Dc\to\Bc_k(\Vc)$, $(P_1,\dots,P_m)\mapsto P_1+\cdots+P_m$,
		\item $\Bc_k(\Vc)\to\Gr_k(\Vc)$, $T\mapsto\Ran T$.  
	\end{itemize}
	where $\Dc:=\{(P_1,\dots,P_m)\in\Bc(\Vc)^m\mid \rank(P_1+\cdots+P_m)=k\}$ with its topology inherited from $\Bc(\Vc)^m$.
	The first of the above three maps is continuous by one of the equivalent descriptions of the topology of $\Gr(\Vc)$, 
	the second map is clearly continuous, and the third map is continuous by Lemma~\ref{rancont}.
This completes the proof. 
\end{proof}

\begin{theorem}\label{Lcont}
	Let $\Vc$ be a vector space over $\KK\in\{\RR,\CC\}$ 
	with $m:=\dim_\KK\Vc<\infty$. 
	Fix a sequence of linear subspaces  
	$\{0\}=\Vc_0\subseteq\Vc_1\subseteq\cdots\subseteq\Vc_m=\Vc$ with $\dim\Vc_j=j$ for $j=0,\dots,m$,  
	and define $B_j:=B\vert_{\Vc_j\times\Vc_j}\in (\wedge^2\Vc_j)^*$ for $j=0,\dots,m$ and $B\in(\wedge^2\Vc)^*$. 
	Define the map 
	$$\pg\colon(\wedge^2\Vc)^*\to\Gr(\Vc),\quad 
	\pg(B):=N(B_1)+\cdots+N(B_m).$$
Set
\begin{equation}
\label{Lcont_eq1}
J_m:=\{\mathbf{k}=(k_1,\dots,k_m)\in\NN^m\mid 0\le k_j\le j\text{ for } j=0,\dots,m\}
\end{equation}	
	and, for every $\mathbf{k}=(k_1,\dots,k_m)\in J_m$, define
\begin{equation*}
(\wedge^2\Vc)^*_{\mathbf{k}}:=\{B\in(\wedge^2\Vc)^*\mid \dim N(B_j)=k_j\text{ for }j=1,\dots,m\}.
\end{equation*}
	Then $\pg(B)\in \Pg(B)$ for for every $B\in (\wedge^2\Vc)^*_{\mathbf{k}}$. 
	Moreover,  the mapping $\pg$ is continuous on every 
	 $S\subseteq(\wedge^2\Vc)^*$ such that the set $S\cap(\wedge^2\Vc)^*_{\mathbf{k}}$ is relatively closed in~$S$
	for every $\mathbf{k}\in J_m$.
\end{theorem}

\begin{proof}
The fact that $\pg(B)\in \Pg (B)$ follows from \cite[Lemma 1.12.3(i)]{Dix74}.

The disjoint union 
	$$S=\bigcup_{\mathbf{k}\in J_m}S\cap(\wedge^2\Vc)^*_{\mathbf{k}}$$
	is a finite partition of $S$ into relatively closed subsets, hence these subsets are also relatively open in~$S$. 
	Therefore it suffices to prove that the mapping $\pg$ is continuous on each of these relatively open subsets. 
	To this end we fix $\mathbf{k}=(k_1,\dots,k_m)\in J_m$ and we prove that the mapping $\pg\vert_{(\wedge^2\Vc)^*_{\mathbf{k}}}\colon (\wedge^2\Vc)^*_{\mathbf{k}}\to\Gr(\Vc)$ is continuous. 
	
	For $j=1,\dots,m$, the map 
	$$\gamma_j\colon(\wedge^2\Vc)^*_{\mathbf{k}}\to\Gr_{k_j}(\Vc_j)\hookrightarrow\Gr_{k_j}(\Vc),\quad 
	B\mapsto N(B_j)$$
	is continuous by Proposition~\ref{nullcont} along with the fact that the restriction mapping 
	$$(\wedge^2\Vc)^*\to(\wedge^2\Vc_j)^*,\quad B\mapsto B_j$$
	is continuous.  
Then note that for $B\in (\wedge^2\Vc)^*_{\mathbf{k}}$, $\dim N(B) = \dim N_m(B_m)=k_m$.
	Since $\pg (B) \in \Pg(B)$, $\dim \pg (B) = (m+k_m)/2=: k$, hence 
	$(\gamma_1(B), \dots, \gamma_m(B))\in \Gr_{k_1,\dots,k_m}^k(\Vc)$.

	Summing up, the map $\pg\vert_{(\wedge^2\Vc)^*_{k_1,\dots,k_m}}$ is the composition of the following continuous maps
	\begin{itemize}
		\item $(\wedge^2\Vc)^*_{\mathbf{k}}\to \Gr_{k_1,\dots,k_m}^k(\Vc),\quad 
		B\mapsto (\gamma_1(B),\dots,\gamma_m(B))$, 
		\item $\Gr_{k_1,\dots,k_m}^k(\Vc)\to \Gr_k(\Vc)$, $(\Vc_1,\dots,\Vc_m)\mapsto\Vc_1+\cdots+\Vc_m$, 
	\end{itemize}
	where the second of the above maps is continuous by Proposition~\ref{addcont}. 
	This completes the proof. 
\end{proof}

\begin{remark}
	\normalfont 
	In Theorem~\ref{Lcont}, 
	for every $B\in(\wedge^2\Vc)^*$, the linear subspace $\pg(B)$ is a Lagrangian subspace for the presymplectic structure $B$ 
	by \cite[Lemma 1.12.3]{Dix74}, that is, 	we have $\pg(B)\in \Pg(B)$. 
	Thus the map $\pg\colon(\wedge^2\Vc)^*\to\Gr(\Vc)$ is a selection of 
	the upper semicontinuous set-valued map 
	$\Pg\colon(\wedge^2\Vc)^*\to 2^{\Gr(\Vc)}$ (see Proposition~\ref{Luppsemi} for $F=\Gr(\Vc)$), 
	and that selection is continuous on every subset $(\wedge^2\Vc)^*_{\mathbf{k}}\subseteq(\wedge^2\Vc)^*$ 
	for arbitrary $\mathbf{k}\in J_m$. 
\end{remark}

\section{Continuous selections and Schubert cells in Grassmann manifolds}
\label{Sect5}

In this section, the maximal continuity domains from Theorem~\ref{Lcont} are described in terms of Schubert cells in Grassmann manifolds. 
To this end we generalize \cite[Lemma 3.2]{Cu88} and \cite[Lemma 1.1]{Cu92}, 
using the relation between jump indices and Schubert cells established in \cite{BB17}.

Throughout this section we keep the notation in Theorem~\ref{Lcont}. 
Namely, $\Vc$ is a vector space over $\KK\in\{\RR,\CC\}$ 
with $m:=\dim_\KK\Vc<\infty$ and $B\in(\wedge^2\Vc)^*\setminus \{0\}$ is 
a presymplectic structure.  
We fix a sequence of linear subspaces  
$$\{0\}=\Vc_0\subsetneqq\Vc_1\subsetneqq\cdots\subsetneqq\Vc_m=\Vc$$
with $\dim_\KK\Vc_j=j$ for $j=0,\dots,m$,  
set $B_j:=B\vert_{\Vc_j\times\Vc_j}\in (\wedge^2\Vc_j)^*$ for $j=0,\dots,m$, 
 and define $\pg(B)\in\Pg(B)$ by
\begin{equation}\label{pgB}
\pg(B):=N(B_1)+\cdots+N(B_m)\in\Gr(\Vc).
\end{equation}
For every $\Wc\in\Gr(\Vc)$ we define 
$$\jump\Wc:=
\{j\in\{1,\dots,m\}\mid \Vc_j\not\subset\Vc_{j-1}+\Wc\}.$$
Then for every subset $e\subseteq\{1,\dots,m\}$ its corresponding \emph{Schubert cell} is 
$$\Gr_e(\Vc):=\{\Wc\in\Gr(\Vc)\mid\jump\Wc=e\}$$
and we note that for every integer $k\in\{1,\dots,m\}$ we have the following finite partition of the set of $k$-dimensional linear subspaces of $\Vc$: 
$$\Gr_k(\Vc)=\bigsqcup_{\card e=m-k}\Gr_e(\Vc).$$
See \cite[Sect. 3]{BB17} for more details. 

Thus, the aim of this section in to prove the following
characterization of the continuity domains of the mapping $\pg\colon(\wedge^2\Vc)^*\to\Gr(\Vc)$, $\pg(B):=N(B_1)+\cdots+N(B_m)$, 
from Theorem~\ref{Lcont} 
in terms of Schubert cells in the Grassmann manifold~$\Gr(\Vc)$.  
\begin{theorem}\label{Sopt3}
	If $S\subseteq(\wedge^2\Vc)^*$, and  
	for every $e\subseteq\{1,\dots,m\}$ the set $S\cap\pg^{-1}(\Gr_e(\Vc))$ is relatively closed in~$S$, 
	then the mapping 
	$\pg\vert_S\colon S\to\Gr(\Vc)$ is continuous. 
\end{theorem}
The remaining part of this section is devoted to the proof of Theorem~\ref{Sopt3}. 

\begin{definition}\label{D1}
\normalfont
We set $\pg^0(B):=\Vc$. 
Inductively, assume $k\ge 0$ is an integer and we have already defined the linear subspaces $\pg^0(B)\supseteq\cdots\supseteq\pg^k(B)$ of~$\Vc$. 
If the condition $\pg^k(B)\not\perp_B\pg^k(B)$ is satisfied, then we define 
\begin{align}
\label{D1_eq1}
i_{k+1}:=
& \min\{i\in\{0,\dots,m\}\mid \Vc_i\cap\pg^k(B)\not\perp_B\pg^k(B)\},\\ 
\label{D1_eq2}
\pg^{k+1}(B):=
&(\Vc_{i_{k+1}}\cap\pg^k(B))^{\perp_B}\cap\pg^k(B)
\end{align}
Moreover, we define 
\begin{equation}
\label{D1_eq3}
j_{k+1}:=\min\{j\in\{0,\dots,m\}\mid \Vc_j\cap\pg^k(B)\not\subset\pg^{k+1}(B)\}.
\end{equation}
\end{definition}

In the following lemma we collect some features of the above inductive construction, 
in particular showing that it eventually stops. 

\begin{lemma}\label{R1}
With the above notation, one has: 
\begin{enumerate}[{\rm(i)}]
	\item\label{R1_item1} 
	$\pg^k(B)\supsetneqq\pg^{k+1}(B)$ and $\dim(\pg^k(B)/\pg^{k+1}(B))=1$;
	\item\label{R1_item2} 
	 $\pg^k(B)=\pg^{k+1}(B)\dotplus(\Vc_{j_{k+1}}\cap\pg^k(B))$; 
	\item\label{R1_item3}  
	$\Vc_{i_{k+1}}\cap\pg^k(B)\subseteq\pg^{k+1}(B)$;
	\item\label{R1_item4}  
	$\Vc_{i_{k+1}}\cap\pg^k(B)\perp_B\pg^{k+1}(B)$; 
	\item\label{R1_item5}
	$\pg^k(B)^{\perp_B}\cap\pg^k(B)\subseteq \pg^{k+1}(B)^{\perp_B}\cap \pg^{k+1}(B)$.
	\end{enumerate}
\end{lemma}

\begin{proof}
	\eqref{R1_item1}
We have $\pg^{k+1}(B)\subseteq\pg^{k}(B)$ by the definition of $\pg^{k+1}(B)$ in \eqref{D1_eq2}. 
Moreover, by the definition of $i_{k+1}$ in \eqref{D1_eq1}, we have
\begin{equation}
\label{R1_proof_eq1}
(\exists x_0\in\Vc_{i_{k+1}}\cap\pg^k(B)) \quad x_0\not\perp_B\pg^k(B)
\end{equation}
and 
\begin{equation}
\label{R1_proof_eq2}
\Vc_{i_{k+1}-1}\cap\pg^k(B)\perp_B\pg^k(B). 
\end{equation}
We now make the general remark:
If $ \Wc_0,\Wc_1,\Wc_2\in\Gr(\Vc)$ and $ \Wc_1\subseteq\Wc_2 $ then
$$
\dim((\Wc_2\cap\Wc_0)/(\Wc_1\cap\Wc_0))\le\dim(\Wc_2/\Wc_1), $$
which follows from the fact that the mapping 
$(\Wc_2\cap\Wc_0)/(\Wc_1\cap\Wc_0)\to\Wc_2/\Wc_1$, $x+(\Wc_1\cap\Wc_0)\to x+\Wc_1$, is well defined, linear and injective. 
This implies by \eqref{R1_proof_eq1}--\eqref{R1_proof_eq2} that 
\begin{equation}
\label{R1_proof_eq3}
\Vc_{i_{k+1}}\cap\pg^k(B)=\RR x_0\dotplus (\Vc_{i_{k+1}-1}\cap\pg^k(B))
\end{equation}
hence 
$$(\Vc_{i_{k+1}}\cap\pg^k(B))^{\perp_B}\cap\pg^k(B)
=\{x_0\}^{\perp_B}\cap \pg^k(B).$$
This is further equivalent to 
$$\pg^{k+1}(B)=\Ker\psi_{x_0}$$
where the linear functional $\psi_{x_0}\colon \pg^k(B)\to\KK$, $\psi_{x_0}(v):=B(x_0,v)$, satisfies $\psi_{x_0}\ne0$ by \eqref{R1_proof_eq1} 
hence $\dim(\pg^k(B)/\Ker\psi_{x_0})=1$, and thus $\dim(\pg^k(B)/\pg^{k+1}(B))=1$. 
	
\eqref{R1_item2}
We have just seen that $\dim(\pg^k(B)/\pg^{k+1}(B))=1$, 
while 
the definition of $j_{k+1}$ in \eqref{D1_eq3} implies 
$\Vc_{j_{k+1}}\cap\pg^k(B)\not\subset\pg^{k+1}(B)$, hence 
$\pg^k(B)=\pg^{k+1}(B)\dotplus(\Vc_{j_{k+1}}\cap\pg^k(B))$. 
	
\eqref{R1_item3}
By \eqref{R1_proof_eq1}--\eqref{R1_proof_eq3} along with $B(x_0,x_0)=0$ we obtain that $\Vc_{i_{k+1}}\cap\pg^k(B)$ is an isotropic subspace with respect to $B$, that is, 
$\Vc_{i_{k+1}}\cap\pg^k(B)\perp_B \Vc_{i_{k+1}}\cap\pg^k(B)$. 
Then, by the definition of $\pg^{k+1}(B)$ in \eqref{D1_eq2}, 
we obtain the assertion. 
	
\eqref{R1_item4}
The definition of $\pg^{k+1}(B)$ in \eqref{D1_eq2} implies 
$\pg^{k+1}(B)\subseteq (\Vc_{i_{k+1}}\cap\pg^k(B))^{\perp_B}$. 

\eqref{R1_item5}
We have $\pg^k(B)^{\perp_B}\cap\pg^k(B)\subseteq
(\Vc_{i_{k+1}}\cap \pg^k(B))^{\perp_B}\cap\pg^k(B)=\pg^{k+1}(B)$ 
and on the other hand $\pg^{k+1}(B)\subseteq\pg^k(B)$, which implies 
$\pg^k(B)^{\perp_B}\subseteq \pg^{k+1}(B)^{\perp_B}$. 
Thus we obtain the inclusion in the statement.
\end{proof}

\begin{definition}
\label{D2}
\normalfont 
We denote by $d\in\{1,\dots,m\}$ the integer that is uniquely determined by the properties
\begin{align*}
\pg^{d-1}(B)& \not\perp_B\pg^{d-1}(B), \\
\pg^d(B)&\perp_B\pg^d(B).
\end{align*} 
The existence of $d$ follows by Lemma~\ref{R1}\eqref{R1_item1}. 
\end{definition}

\begin{lemma}
\label{L1}
We have  $d=\dim(\Vc/\pg(B))=\frac{1}{2}\dim(\Vc/N(B))$ 
and  $\pg^d(B)=\pg(B)\in\Pg(B)$. 
\end{lemma}

\begin{proof}
We have 
 $\pg(B)\in\Pg(B)$, thus $d=\dim(\Vc/\pg(B))=\frac{1}{2}\dim(\Vc/N(B))$, 
as noted after \eqref{pgB}. 

It remains to prove the equality $\pg^d(B)=\pg(B)$. 

``$\supseteq$'' 
We first prove the following inclusion for $i=1,\dots,m$: 
\begin{equation}
\label{L1_proof_eq1}
N(B_i)\cap\pg^k(B)\subseteq\pg^{k+1}(B) \text{ if }0\le k<d.
\end{equation}
To prove the above inclusion we separately discuss the two cases that can occur: 

Case 1: $i<i_{k+1}$. 
Then $\Vc_i\cap\pg^k(B)\perp_B\pg^k(B)$ by the definition of $i_{k+1}$ in \eqref{D1_eq1}, hence $N(B_i)\cap\pg^k(B)\perp_B\pg^k(B)$, 
and then 
$$N(B_i)\cap\pg^k(B) \subseteq\pg^k(B)^{\perp_B}\cap\pg^k(B)\subseteq\pg^{k+1}(B)$$  
where the last inclusion follows by the definition of $\pg^{k+1}(B)$ in \eqref{D1_eq2}. 

Case 2: $i_{k+1}\le i$. 
One then has 
\begin{align*}
N(B_i)\cap\pg^k(B)
& \subseteq \Vc_i^{\perp_B}\cap\pg^k(B) 
\subseteq  (\Vc_i\cap\pg^k(B))^{\perp_B}
\cap\pg^k(B)  \\
& \subseteq  (\Vc_{i_{k+1}}\cap\pg^k(B))^{\perp_B}\cap\pg^k(B)
=\pg^{k+1}(B). 
\end{align*}
A repeated application of \eqref{L1_proof_eq1} gives
\begin{align*} 
& N(B_i)\subseteq\Vc=\pg^0(B),\\
& N(B_i)=N(B_i)\cap\pg^0(B)\subseteq N(B_i)\cap\pg^1(B)\subseteq\cdots\subseteq N(B_i)\cap\pg^d(B)\subseteq \pg^d(B).
\end{align*}
Since $\pg(B)=N(B_1)+\cdots+N(B_m)$, we thus obtain $\pg^d(B)\supseteq\pg(B)$. 

``$\subseteq$'' 
By Definition~\ref{D2} we have $\pg^d(B)\perp_B\pg^d(B)$, 
hence $\pg^d(B)\in\Sbd(B)$. 
On the other hand, $\pg(B)\in\Pg(B)$, as noted at the beginning of the present proof, 
hence $\dim\pg^d(B)\le\dim\pg(B)$. 
Thus, since we already proved the inclusion $\pg^d(B)\supseteq\pg(B)$, 
we obtain $\pg^d(B)=\pg(B)$, and this completes the proof. 
\end{proof}

\begin{lemma}
	\label{L2}
	We have $i_k,j_k\in\jump N(B)$. Moreover $i_k<j_k$ and $i_k<i_{k+1}$. 
\end{lemma}

\begin{proof}
If $i_k\not\in\jump N(B)$, then $\Vc_{i_k-1}\subseteq\Vc_{i_k}\subseteq\Vc_{i_k-1}+N(B)$, 
hence 
$$\Vc_{i_k}=\Vc_{i_k-1}+(\Vc_{i_k}\cap N(B)).$$ 
On the other hand 
\begin{equation}
\label{L2_proof_eq1}
N(B)=N(B_m)\subseteq\pg(B)=\pg^d(B)\subseteq\pg^{k-1}(B)
\end{equation}
(where the second equality follows by Lemma~\ref{L1}) 
hence 
\begin{equation}
\label{L2_proof_eq2}
\Vc_{i_k}\cap\pg^{k-1}(B)=(\Vc_{i_k-1}\cap\pg^{k-1}(B))+(\Vc_{i_k}\cap N(B)).
\end{equation}
By the definition of $i_k$ we have $\Vc_{i_k-1}\cap\pg^{k-1}(B)\perp_B\pg^{k-1}(B)$ hence, 
by \eqref{L2_proof_eq2}, we obtain 
$\Vc_{i_k}\cap\pg^{k-1}(B)\perp_B\pg^{k-1}(B)$, 
which is a contradiction with the definition of $i_k$. 
Consequently $i_k\in\jump N(B)$. 

If $j_k\not\in\jump N(B)$, then $\Vc_{j_k-1}\subseteq\Vc_{j_k}\subseteq\Vc_{j_k-1}+N(B)$, 
hence 
$$\Vc_{j_k}=\Vc_{j_k-1}+(\Vc_{j_k}\cap N(B)).$$ 
Then, by \eqref{L2_proof_eq1}, 
\begin{equation}
\label{L2_proof_eq3}
\Vc_{j_k}\cap\pg^{k-1}(B)=(\Vc_{j_k-1}\cap\pg^{k-1}(B))+(\Vc_{j_k}\cap N(B)).
\end{equation}
On the other hand, by the definition of $j_k$, we have $\Vc_{j_k-1}\cap\pg^{k-1}(B)\subseteq\pg^k(B)$ hence, by \eqref{L2_proof_eq3}, 
$$\Vc_{j_k}\cap\pg^{k-1}(B)\subseteq\pg^k(B)+N(B)\subseteq \pg^k(B)+\pg(B)
=\pg^k(B)+\pg^d(B)\subseteq\pg^k(B)$$
(where the equality follows by Lemma~\ref{L1}), and we thus obtained a contradiction with the definition of $j_k$. 
Consequently $j_k\in\jump N(B)$. 

We now prove that $i_k<j_k$. 
To this end, by the definition of $j_k$, it suffices to show that 
$$ 
\Vc_{i_k}\cap\pg^{k-1}(B)\subseteq\pg^k(B). 
$$
But this follows from 
Lemma~\ref{R1}\eqref{R1_item3}, applied for $k-1$ instead of $k$.

It remains to prove that $i_k<i_{k+1}$. 
To this end, by the definition of $i_{k+1}$, it suffices to prove that $\Vc_{i_k}\cap\pg^k(B)\perp_B\pg^k(B)$. 
In fact, by the definition of $\pg^k(B)$, we obtain 
\begin{equation}
\label{L2_proof_eq6}
\pg^k(B)= (\Vc_{i_k}\cap\pg^{k-1}(B))^{\perp_B}\cap\pg^{k-1}(B)
\end{equation}
which further implies 
$$\Vc_{i_k}\cap\pg^k(B)= 
(\Vc_{i_k}\cap\pg^{k-1}(B))^{\perp_B}\cap(\Vc_{i_k}\cap\pg^{k-1}(B)).$$
Since $\pg^k(B)\subseteq\pg^{k-1}(B)$, 
it then follows that $\Vc_{i_k}\cap\pg^k(B)=\Vc_{i_k}\cap\pg^{k-1}(B)$, 
and on the other hand $\Vc_{i_k}\cap\pg^{k-1}(B)\perp_B\pg^k(B)$ by \eqref{L2_proof_eq6}, hence finally $\Vc_{i_k}\cap\pg^k(B)\perp_B\pg^k(B)$. 
This completes the proof. 
\end{proof}

\begin{lemma}
\label{L3}
The mapping $\{1,\dots,d\}\to\jump N(B)\setminus\jump \pg(B)$, $k\mapsto i_k$, is a well-defined increasing bijection. 
\end{lemma}

\begin{proof}
It follows by Lemma~\ref{L2} that $k\mapsto i_k$ is an increasing mapping, 
hence injective. 
On the other hand,  $\jump \pg(B)\subseteq\jump N(B)$ since $N(B)\subseteq\pg(B)$. 
Moreover $\card (\jump \pg(B))=\dim(\Vc/\pg(B))$ and $\card (\jump N(B))=\dim(\Vc/N(B))$ (see, e.g.,  \cite[Prop. 3.4(xi)]{BB17}) 
hence $\card (\jump \pg(B))=d$ and $\card (\jump N(B))=2d$ by Lemma~\ref{L1}. 
Therefore 
$$\card(\jump N(B)\setminus\jump \pg(B))=d,$$ 
and it remains to prove that the injective mapping $k\mapsto i_k$ indeed takes values in the set $\jump N(B)\setminus\jump \pg(B)$. 
We already know that $i_k\in\jump N(B)$ by Lemma~\ref{L2}, 
so it remains to prove that $i_k\not\in\jump \pg(B)$. 

Iterating the inclusion in Lemma~\ref{R1}\eqref{R1_item5} for $k=1,\dots,d-1$, 
and using the fact that $\pg^d(B)=\pg^d(B)^{\perp_B}=\pg(B)\in\Pg(B)$ 
(by Lemma~\ref{L1}), we obtain that for $1\le k\le d$, 
 \begin{equation}
\label{L3_proof_eq1}
\pg^k(B)^{\perp_B}\cap\pg^k(B)\subseteq \pg^d(B)=\pg(B). 
\end{equation}
 
We now prove that $i_k\not\in\jump \pg(B)$, that is, 
\begin{equation}
\label{L3_proof_eq2}
\Vc_{i_k}\subseteq\Vc_{i_k-1}+\pg(B). 
\end{equation}
In fact, by \eqref{R1_proof_eq1}--\eqref{R1_proof_eq2} 
in the proof of Lemma~\ref{R1} (applied for $k-1$ instead of $k$), there is
$x_0\in\Vc_{i_k}\cap\pg^{k-1}(B))$ with $x_0\not\perp_B\pg^{k-1}(B)$ 
and 
$\Vc_{i_k-1}\cap\pg^{k-1}(B)\perp_B\pg^{k-1}(B)$. 
In particular $x_0\in\Vc_{i_k}\setminus\Vc_{i_k-1}$, 
hence $\Vc_{i_k}=\RR x_0\dotplus \Vc_{i_k-1}$. 
Therefore, to complete the proof of~\eqref{L3_proof_eq2}, 
it remains to show that $x_0\in\pg(B)$. 

To this end, recall from the proof of the inequality $i_k<i_{k+1}$ in Lemma~\ref{L2} that 
$\Vc_{i_k}\cap\pg^k(B)=\Vc_{i_k}\cap\pg^{k-1}(B)\perp_B\pg^k(B)$. 
Since $x_0\in \Vc_{i_k}\cap\pg^k(B)$, we then obtain $x_0\in\pg^k(B)^{\perp_B}\cap\pg^k(B)$, hence $x_0\in\pg(B)$ by \eqref{L3_proof_eq1}, and this completes the proof. 
\end{proof}

\begin{lemma}
	\label{L4}
	The mapping $\{1,\dots,d\}\to\jump \pg(B)$, $k\mapsto j_k$, is a well-defined bijection. 
\end{lemma}

\begin{proof}
One has $\card(\jump\pg(B))=d$ (cf. the beginning of the proof of Lemma~\ref{L3}) hence it suffices to prove the equality 
\begin{equation}
\label{L4_proof_eq1}
\{j_k\mid 1\le k\le d\}=\jump\pg(B).
\end{equation}
We prove this by double inclusion. 

``$\subseteq$'' 
If $j_k\not\in \jump\pg(B)$ then $\Vc_{j_k}\subseteq\Vc_{j_k-1}+\pg(B)=\Vc_{j_k-1}+\pg^d(B)\subseteq \Vc_{j_k-1}+\pg^k(B)$ (where the equality follows by Lemma~\ref{L1}) 
hence 
$$\Vc_{j_k}\cap \pg^{k-1}(B)\subseteq (\Vc_{j_k-1}+\pg^k(B))\cap \pg^{k-1}(B)
=(\Vc_{j_k-1}\cap \pg^{k-1}(B))+\pg^k(B).$$
By the definition of $j_k$ we have 
$\Vc_{j_k-1}\cap \pg^{k-1}(B)\subseteq\pg^k(B)$ hence the above inclusion 
implies $\Vc_{j_k}\cap \pg^{k-1}(B)\subseteq \pg^k(B)+\pg^k(B)=\pg^k(B)$, 
which is a contradiction with the definition of $j_k$. 

``$\supseteq$'' 
Let $j\in \jump\pg(B)$ be arbitrary. 
Then $\Vc_j\not\subset\Vc_{j-1}+\pg(B)$. 
Let us define 
$$k_0:=\max\{k\in\{0,\dots,d\}\mid \Vc_j\subseteq\Vc_{j-1}+\pg^k(B)\}$$
so that $k_0\le d-1$ since $\Vc_j\not\subset\Vc_{j-1}+\pg(B)=\Vc_{j-1}+\pg^d(B)$ 
 (where the equality follows by Lemma~\ref{L1} again). 
We will prove the equality 
\begin{equation*}
j=j_{k_0+1}.
\end{equation*}
In fact, by the definition of $k_0$ we have 
\begin{equation}
\label{L4_proof_eq3}
\Vc_j\not\subset\Vc_{j-1}+\pg^{k_0+1}(B)
\end{equation}
and 
\begin{equation}
\label{L4_proof_eq4}
\Vc_{j-1}\subseteq \Vc_j\subseteq\Vc_{j-1}+\pg^{k_0}(B).
\end{equation} 
Therefore $\Vc_j=\Vc_{j-1}+(\Vc_j\cap\pg^{k_0}(B))$. 
It then follows by \eqref{L4_proof_eq3} that 
$\Vc_j\cap\pg^{k_0}(B)\not\subset\pg^{k_0+1}(B)$ 
hence, by the definition of $j_{k_0+1}$, we obtain $j_{k_0+1}\le j$. 

If we assume $j_{k_0+1}\le j-1$ then 
\begin{eqnarray*}
\Vc_j
\subseteq
 & \Vc_{j-1}+\pg^{k_0}(B)  \phantom{AAAAAAA}
\qquad &(\text{by \eqref{L4_proof_eq4}})\\
\subseteq &\Vc_{j-1}+\pg^{k_0+1}(B)+\Vc_{j_{k_0+1}} 
\qquad &(\text{by Lemma~\ref{R1}\eqref{R1_item2}})\\
=&\Vc_{j-1}+\pg^{k_0+1}(B) \phantom{AAAAAi}
\qquad &(\text{by the assumption $j_{k_0+1}\le j-1$})
\end{eqnarray*}
 which is a contradiction with the maximality condition in the definition of $k_0$. 
 This completes the proof of the equality $j=j_{k_0+1}$, hence of the inclusion $\supseteq$ in \eqref{L4_proof_eq1} since $j\in\jump\pg(B)$ is arbitrary. 
\end{proof}

\begin{lemma}
\label{L5}
Let the integers $1\le r_1<\cdots<r_{m-d}<r_{m-d+1}=m+1$ satisfy 
$$\{1,\dots,m\}\setminus\jump\pg(B)=\{r_1,\dots,r_{m-d}\}.$$
Then 
\begin{enumerate}[\rm (i)]
\item $2\ell-j \ge 0$ whenever $r_\ell\le j<r_{\ell+1}$ for $\ell=1,\dots,m-d$. 
\item $B\in (\wedge^2\Vc)^*_{\mathbf k}$, where ${\mathbf k}=(k_1, k_2, \dots, k_m)$, with 
$$ k_j = 
\begin{cases} 
j & \text{  if } 1\le j < r_1,\\
2\ell -j  & \text{ if }  r_\ell\le j<r_{\ell+1}  \text{ for } \ell=1,\dots,m-d.
\end{cases}
$$
\end{enumerate}
 \end{lemma}

\begin{proof}
	We first note that the definition of the integers $ r_1<\cdots<r_{m-d}$ makes sense by Lemma~\ref{L4}. 
Moreover, we have by \cite[Prop. 3.4((viii)--(ix))]{BB17} 
\begin{equation}
\label{L5_proof_eq1}
\dim(\pg(B)\cap\Vc_j)=\begin{cases}  j & \text {if } 1 \le j < r_1\\
\ell & \text{ if }r_\ell\le j<r_{\ell+1}
\end{cases}
\end{equation}
On the other hand, it follows by \cite[Lemma 1.12.3]{Dix74} that  $\pg(B)\cap\Vc_j$ is a Lagrangian subspace for the presymplectic structure $B_j=B\vert_{\Vc_j\times\Vc_j}$ on $\Vc_j$, hence $\dim(\Vc_j/N(B_j))=2\dim(\Vc_j/(\pg(B)\cap\Vc_j))$. 
This further implies $\dim\Vc_j-\dim N(B_j)=2\dim\Vc_j-2\dim(\pg(B)\cap\Vc_j)$, 
hence 
$\dim N(B_j)=2\dim(\pg(B)\cap\Vc_j)-\dim\Vc_j$. 
Now the assertion follows by~\eqref{L5_proof_eq1}. 
\end{proof}

\begin{proof}[Proof of Theorem~\ref{Sopt3}]
	Use Theorem~\ref{Lcont} and Lemma~\ref{L5}. 
\end{proof}

\section{Application to completely solvable Lie algebras}
\label{Sect6}

Throughout this section, for an arbitrary finite-dimensional Lie algebra $\gg$ over~$\KK$
we consider the linear mapping 
$$\beta\colon\gg^*\to(\wedge^2\gg)^*,\quad \beta(\xi):=B_\xi,$$
where $B_\xi(x,y):=\langle\xi,[x,y]\rangle$ for all $x,y\in\gg$ and $\xi\in\gg^*$. 
We recall that the Lie algebra $\gg$ is called \emph{completely solvable} if it admits a Jordan-H\"older sequence, that is, an increasing sequence of ideals of $\gg$, 
$$\{0\}=\gg_0\subsetneqq\gg_1\subsetneqq\cdots\subsetneqq\gg_m=\gg$$
with $\dim\gg_j=j$ for $j=0,1,\dots,m$. 
Equivalently, $\gg$ is completely solvable if and only if it is a solvable Lie algebra and all the eigenvalues of the linear mapping $\ad_\gg x\colon\gg\to\gg$, $(\ad_\gg x)y=[x,y]$, belong to $\KK$ (rather than to an algebraic closure of $\KK$ )for every $x\in\gg$. 
If the field $\KK$ is algebraically closed, (e.g., $\KK=\CC$), then $\gg$ is a completely solvable Lie algebra if and only if it is a solvable Lie algebra, by Sophus Lie's classical theorem on representations of solvable Lie algebras. 

The main result of this section (Theorem~\ref{polcont}) establishes continuity properties of the \emph{Vergne mapping} 
\begin{equation}
\label{palg}
\pg_{\alg}:=\pg\circ\beta\colon\gg^*\to\Gr(\gg),\quad 
\pg_{\alg}(\xi)=\gg_1(\xi\vert_{\gg_1})+\cdots+\gg_{m-1}(\xi\vert_{\gg_{m-1}})+\gg_m(\xi)
\end{equation}
associated to a Jordan-H\"older sequence in a completely solvable Lie algebra $\gg$ over $\KK\in\{\RR,\CC\}$ as above, that maps $\xi\in\gg^*$ to a polarization at $\xi$, called the Vergne polarization. 
This construction goes back to \cite{Ve70}. 
The continuity domains of the mapping $\pg_{\alg}$ will be described in terms of the sets 
\begin{equation}
\label{Xiks}
\Xi_{\mathbf{k}}:=\{\xi\in\gg^*\mid \dim\gg_j(\xi\vert_{\gg_j})=k_j\text{ for }j=1,\dots,m\}
\end{equation}
defined for any $\mathbf{k}=(k_1,\dots,k_m)\in J_m$, where we use the notation \eqref{Lcont_eq1}.

\subsection*{Continuity of isotropy groups}

\begin{proposition}\label{isocont}
	Let $G$ be a 
	Lie group of dimension $m\ge 1$, 
	with the quotient map of its corresponding coadjoint action denoted by 
	$q\colon\gg^*\to\gg^*/G$.  
	Let $(\gg^*/G)_d$ be the set of all coadjoint orbits of some fixed dimension~$d\ge 0$, 
	and denote $\Xi_d:=q^{-1}((\gg^*/G)_d)\subseteq\gg^*$, 
	the set of all functionals in $\gg^*$ with $d$-dimensional coadjoint orbits. 
	Then, for any subset $\Xi\subseteq \gg^*$,  the mapping
	$$\Xi\to\Gr(\gg),\quad \xi\mapsto \gg(\xi)$$
	is continuous if and only if for every even integer $d\in\{0,\dots,m\}$ the subset $\Xi\cap\Xi_d$ is relatively closed in~$\Xi$.  
\end{proposition}

\begin{proof}
If for every even integer $d\in\{0,\dots,m\}$ the subset $\Xi\cap\Xi_d$ is relatively closed in~$\Xi$, then the disjoint union 
$$\Xi=\bigsqcup_{d\in2\NN\cap \{0,\dots,m\}}\Xi\cap\Xi_d$$
is a finite partition of $\Xi$ into relatively closed subsets, hence each of these subsets is also relatively open. 
Hence it suffices to prove that the mapping $\xi\mapsto\gg(\xi)$ is continuous on $\Xi\cap\Xi_d$ for every even integer $d\in\{0,\dots,m\}$. 
To this end we fix $d\in\{0,\dots,m\}$ and we claim that the mapping $\Xi_d\to\Gr(\gg)$, $\xi\mapsto\gg(\xi)$ is continuous. 
In fact, for every $\xi\in\gg^*$ we have  
\begin{equation}
\label{isocont_proof_eq1}
\gg(\xi)=N(B_\xi)=N(\beta(\xi)).
\end{equation}
Moreover, the mapping 
$$N\circ\beta\vert_{\Xi_d}\colon\Xi_d\to\Gr_{m-d}(\gg),\quad \xi\mapsto \Ker(\beta(\xi))$$
is continuous 
 as the composition of the continuous maps $\beta$ and $N\vert_{(\wedge^2\gg)^*_d}$ (cf. Proposition~\ref{nullcont})	
and this proves our claim. 

Conversely, let us assume that  the mapping
$\Xi\to\Gr(\gg)$, $\xi\mapsto \gg(\xi)$ is continuous. 
We prove that for arbitrary $d\in2\NN\cap\{0,\dots,m\}$ the set 
$\Xi\cap\Xi_d$ is relatively closed in~$\Xi$. 
To this end we claim that if $\xi\in\Xi$ and $\{\xi^{(i)}\}_{i\in\NN}$ is a sequence in $\Xi\cap\Xi_d$ with $\lim\limits_{k\to\infty}\xi^{(i)}=\xi$ in $\gg^*$, then $\xi\in\Xi_d$. 
In fact, since the mapping $\Xi\to\Gr(\gg)$, $\xi\mapsto \gg(\xi)$ is assumed to be continuous, we obtain $\lim\limits_{k\to\infty}\gg(\xi^{(i)})=\gg(\xi)$ in $\Gr(\gg)$. 
Here $\gg(\xi^{(i)})\in\Gr_{m-d}(\gg)$ since $\xi^{(i)}\in\Xi_d$ for every $i\in\NN$. 
Since $\Gr_{m-d}(\gg)$ is a closed subset of $\Gr(\gg)$, it follows that $\gg(\xi)\in\Gr_{m-d}(\gg)$, that is, $\xi\in\Xi_d$, as claimed, 
and this completes the proof. 
\end{proof}

\subsection*{Continuity properties of polarizations}

Let $\gg$ be any finite-dimensional real Lie algebra. 
We define $\Gr_{\alg}(\gg)$ 
	as the set of all subalgebras of $\gg$. 
	As proved in \cite[1.11.9]{Dix74}, $\Gr_{\alg}(\gg)$ is a Zariski-closed subset of 
	the Grassmann manifold $\Gr(\gg)$. 

\begin{definition}\label{polarizations}
	\normalfont
	Let $\gg$ be any finite-dimensional real Lie algebra. 
	
	For every $\xi\in\gg^*$ we define 
	$$\begin{aligned}
	\Sbd(\xi)
	&:=\{\hg\in \Gr_{\alg}(\gg)\mid [\hg,\hg]\subseteq\Ker\xi\}, \\
	\Pg(\xi)
	&:=\{\hg\in \Gr_{\alg}(\gg)\mid \hg \text{ maximal element of }\Sbd(\xi)\} \\
	\end{aligned}$$
	hence $\Sbd(\xi)$ is the set of all subordinated subalgebras and $\Pg(\xi)$ is the set of all polarizations at~$\xi$. 
\end{definition}

\begin{remark}\label{gralg}
	\normalfont
	Taking $F:=\Gr_{\alg}(\gg)$,
	the relation between Definitions \ref{polarizations} and \ref{Lpolarizations} (with $F=\Gr_{\alg}(\gg)$) is given by 
	$$\Sbd(\xi)=\Sbd_F(B_\xi)\text{ and } 
	\Pg(\xi)=\Pg_F(B_\xi)$$
for all $\xi\in\gg^*$. 
\end{remark}

\begin{proposition}\label{uppsemi}
	Let $k\ge 0$ be any integer 
	and 
	$\Xi_k:=\{\xi\in\gg^*\mid \dim\gg(\xi)=k\}\subseteq\gg^*$. 
	Then the map 
	$\Pg\vert_{\Xi_k}\colon\Xi_k\to 2^{\Gr_{\alg}(\gg)}$ 
	is upper semicontinuous. 
\end{proposition}

\begin{proof}
Recall that $\gg(\xi)=N(B_\xi)$ for all $\xi\in\gg^*$, as noted in the proof of Proposition~\ref{isocont}. 
Then use Remark~\ref{gralg} and Proposition~\ref{Luppsemi}, as well as the fact that the mapping $\beta\colon\gg^*\to(\wedge^2\gg)^*$ is continuous. 
\end{proof}

Assume now that $\gg$ is a  
	completely solvable Lie algebra over $\KK\in\{\RR,\CC\}$ with a Jordan-H\"older series 
	$\{0\}=\gg_0\subsetneqq\gg_1\subsetneqq\cdots\subsetneqq\gg_m=\gg$,  
	and the corresponding Vergne mapping 
	$\pg_{\alg}\colon\gg^*\to\Gr(\gg)$. 
In fact, we have the commutative diagram 
$$\xymatrix{  (\wedge^2\gg)^* \ar[r]^{\pg} & \Gr(\gg) \\
	\gg^* \ar[r]_{\pg_{\alg}} \ar[u]^{\beta} & \Gr_{\alg}(\gg) \ar@{^{(}->}[u]
} .$$
Here we have used that, since $\gg$ is a completely solvable Lie algebra, $\pg_{\alg}(\xi)$ is a subalgebra of $\gg$, that is, $\pg_{\alg}(\xi)\in\Gr_{\alg}(\gg)$ for every $\xi\in\gg^*$;
see for instance \cite[Prop. 1.12.10]{Dix74}.

In the following theorem we use notation introduced in \eqref{palg} and \eqref{Xiks}. 

\begin{theorem}\label{polcont}
	Let $\gg$ be a  
	completely solvable Lie algebra over $\KK\in\{\RR,\CC\}$ with a Jordan-H\"older series 
	$\{0\}=\gg_0\subsetneqq\gg_1\subsetneqq\cdots\subsetneqq\gg_m=\gg$,  
	and the corresponding Vergne mapping 
	$\pg_{\alg}\colon\gg^*\to\Gr_{\alg}(\gg)$. 
	If $\Xi\subseteq\gg^*$ is a subset with the property that $\Xi\cap\Xi_{\mathbf{k}}$ is a relatively closed subset of $\Xi$ for every $\mathbf{k}\in J_m$, then the mapping $\pg_{\alg}\vert_\Xi\colon\Xi\to\Gr_{\alg}(\gg)$ is continuous. 
\end{theorem}

\begin{proof}
The disjoint union 
$$\Xi=\bigcup_{\mathbf{k}\in J_m}\Xi\cap\Xi_{\mathbf{k}}$$
is a finite partition of $\Xi$ into relatively closed subsets, hence these subsets are also relatively open in~$\Xi$. 
Therefore it suffices to prove that the mapping $\pg_{\alg}$ is continuous on each of these relatively open subsets. 
To this end we fix $\mathbf{k}=(k_1,\dots,k_m)\in J_m$ and we claim that the mapping $\pg_{\alg}\vert_{\Xi_{\mathbf{k}}}\colon \Xi_{\mathbf{k}}\to\Gr(\gg)$ is continuous. 

Indeed, we note that, by \eqref{Xiks} and \eqref{isocont_proof_eq1},  
$$\Xi_{\mathbf{k}}
=\{\xi\in\gg^*\mid \dim N(\beta(\xi)\vert_{\gg_j\times\gg_j})=k_j\text{ for }j=1,\dots,m\}
=\beta^{-1}((\wedge^2\gg)^*_{\mathbf{k}}).$$
We then obtain the commutative diagram 
$$\xymatrix{ & (\wedge^2\gg)^*_{\mathbf{k}} \ar[d]^{\pg\vert_{ (\wedge^2\gg)^*_{\mathbf{k} }}}\\
\Xi_{\mathbf{k}} \ar[r]_{\pg_{\alg}\vert_{\Xi_{\mathbf{k}}}} \ar[ur]^{\beta\vert_{\Xi_{\mathbf{k}}}} & \Gr(\gg)}$$
where $\beta\vert_{\Xi_{\mathbf{k}}}$ is continuous since 
$\beta\colon\gg^*\to(\wedge^2\gg)^*$ is continuous, 
while $\pg\vert_{ (\wedge^2\gg)^*_{\mathbf{k} }}$ is continuous by Theorem~\ref{Lcont}. 
Thus $\pg_{\alg}\vert_{\Xi_{\mathbf{k}}}\colon \Xi_{\mathbf{k}}\to\Gr(\gg)$ is a composition of two continuous mappings, hence it is in turn continuous. 
This proves our claim, and we are done. 
\end{proof}

\begin{remark}
	\normalfont 
	In Theorem~\ref{polcont}, the Vergne mapping $\pg_{\alg}\colon\gg^*\to\Gr_{\alg}(\gg)$ is actually a distinguished selection of 
	the upper semicontinuous map 
	$\Pg\colon\gg^*\to 2^{\Gr_{\alg}(\gg)}$ (see Proposition~\ref{uppsemi}), 
	and that selection $\pg_{\alg}(\cdot)$ is continuous on every set $\Xi_{\mathbf{k}}\subseteq\gg^*$ 
	for arbitrary $\mathbf{k}\in J_m$. 
\end{remark}

\section{Examples}
\label{Sect7}

\subsection*{Lie algebras with abelian hyperplane ideals}

\begin{example}
\label{ax+b}
Let $\gg$ be a completely solvable Lie algebra over $\KK$ which has an abelian ideal $\ag\subsetneqq\gg$ with $\dim(\gg/\ag)=1$. 
Then there exists a Jordan-H\"older sequence 
$$\{0\}=\gg_0\subsetneqq\gg_1\subsetneqq\cdots\subsetneqq\gg_m=\gg$$
with $\gg_{m-1}=\ag$. 

We compute the corresponding Vergne polarization mapping 
$$(\forall\xi\in\gg^*)\quad 
\pg_{\alg}(\xi)=\gg_1(\xi\vert_{\gg_1})+\cdots
+\gg_{m-1}(\xi\vert_{\gg_{m-1}})+\gg(\xi).$$
To this end we note that for $j=1,\dots,m-1$ the Lie algebra $\gg_j$ is is abelian since it is contained in the abelian ideal $\ag$,  
hence for arbitrary $\xi\in\gg^*$ we have $\gg_j(\xi\vert_{\gg_j})=\gg_j$. 
Therefore 
\begin{equation}
\label{ax+b_eq1}
\pg_{\alg}(\xi)=\ag+\gg(\xi).
\end{equation}
If $\xi\in[\gg,\gg]^\perp$ then $\gg(\xi)=\gg$. 
Now let us assume 
 $\xi\in\gg^*\setminus[\gg,\gg]^\perp$, hence $\gg(\xi)\subsetneqq\gg$. 
We have $\dim(\gg/\pg_{\alg}(\xi))=\dim(\pg_{\alg}(\xi)/\gg(\xi))$ since $\pg_{\alg}(\xi)\in\Pg(\xi)$. 
If $\gg(\xi)\not\subset\ag$ then, by the hypothesis $\dim(\gg/\ag)=1$ along with  \eqref{ax+b_eq1}, we obtain $\pg_{\alg}(\xi)=\gg$. 
On the other hand we know that $\pg_{\alg}(\xi)\in\Pg(\xi)$, hence 
$\dim(\pg_{\alg}(\xi)/\gg(\xi))=\dim(\gg/\pg_{\alg}(\xi))=0$, 
and thus $\gg(\xi)=\pg_{\alg}(\xi)=\gg$, which is a contradiction with the fact that $\xi\in\gg^*\setminus[\gg,\gg]^\perp$. 
Thus if $\xi\in\gg^*\setminus[\gg,\gg]^\perp$ then $\gg(\xi)\subseteq\ag$. 
Then, by \eqref{ax+b_eq1}, we obtain 
\begin{equation}
\label{ax+b_eq2}
\pg_{\alg}(\xi)=
\begin{cases}
\ag &\text{ if }\xi\in\gg^*\setminus[\gg,\gg]^\perp,\\
\gg &\text{ if }\xi\in[\gg,\gg]^\perp. 
\end{cases}
\end{equation}
This illustrates Theorem~\ref{polcont}, which establishes the continuity property of the Vergne polarization mapping $\pg_{\alg}\colon\gg^*\to\Gr_{\alg}(\gg)$  on each of the subsets 
\begin{equation*}
\Xi_{(k_1,\dots,k_m)}=\{\xi\in\gg^*\mid \dim(\gg_j(\xi\vert_{\gg_j}))=k_j\text{ for }j=1,\dots,m\}
\end{equation*}
where $1\le k_j\le j$ for $j=1,\dots,5$. 
In fact, by \eqref{ax+b_eq2}, the Vergne polarization mapping is constant on these subsets of~$\gg^*$. 

By the above computation, if $[\gg,\gg]\ne\{0\}$, then we have $\Xi_{(k_1,\dots,k_m)}\ne\emptyset$ 
if and only if $k_j=j$ for $j=1,\dots,m-1$ and 
$k_m\in\{2,m\}$. 
More specifically, 
\begin{equation*}
\Xi_{(1,\dots,m-1,k_m)}
=\begin{cases}
\gg^*\setminus[\gg,\gg]^\perp &\text{ if }k_m=2,\\
[\gg,\gg]^\perp &\text{ if }k_m=m.
\end{cases}
\end{equation*} 
In particular, the Vergne polarization mapping $\pg_{\alg}\colon\gg^*\to\Gr_{\alg}(\gg)$ is continuous on the open dense subset $\Xi_{(1,\dots,m-1,2)}\subseteq\gg^*$. 
\end{example}

In Example~\ref{ax+b} it turned out that the Vergne polarization mapping is continuous on the maximal domain $\gg^*\setminus[\gg,\gg]^\perp$. 
Therefore the complicated nature of the unitary dual space of the Lie groups of this type (for $\KK=\RR$) is due to the highly non-Hausdorff topology of the quotient topological space of 2-dimensional coadjoint orbits 
$(\gg^*\setminus[\gg,\gg]^\perp)/G$. 
See for instance the threadlike nilpotent Lie groups studied in \cite{ArSoKaSc99}.

We now examine an example that presents a different behaviour: the quotient topological space $(\gg^*\setminus[\gg,\gg]^\perp)/G$ is Hausdorff, 
but the Vergne polarization mapping is not continuous on the whole set  $\gg^*\setminus[\gg,\gg]^\perp$. 

\subsection*{The nilpotent Lie algebra $\gg_{6,15}$} 

\begin{example}
	\label{g615}
	\normalfont
	We consider the 2-step nilpotent real Lie algebra $\gg$ denoted by $\gg_{6,15}$ in \cite{Ni83}, defined by a basis $X_1,X_2,X_3,X_4,X_5,X_6$ satisfying the commutation relations 
	$$[X_6,X_5]=X_3,\ [X_6,X_4]=X_1,\ [X_5,X_4]=X_2.$$
	We consider the Jordan-H\"older sequence 
	$$\gg_0=\{0\}\subsetneqq\gg_1\subsetneqq\gg_2\subsetneqq\gg_3
	\subsetneqq\gg_4\subsetneqq\gg_5\subsetneqq\gg_6=\gg,$$
	where $\gg_j=\spa\{X_i\mid 1\le i\le j\}$ for $j=1,\dots,6$. 
	Then the center of $\gg$ is $\zg:=\gg_3=[\gg,\gg]$. 
	For every $\xi\in\gg^*$ we denote $\xi_j:=\langle\xi,X_j\rangle$ for $j=1,\dots,6$. 
	Similarly, for $\eta\in[\gg,\gg]^*=\gg_3^*$ we denote $\eta_j:=\langle\eta,X_j\rangle$ for $j=1,2,3$. 
	With this notation we define the mapping 
	$$Y\colon[\gg,\gg]^*\to\gg,\quad Y(\eta):=\eta_3X_4-\eta_1X_5+\eta_2X_6$$
	and then we have, cf. \cite{Ni83}, 
	\begin{equation}
	\label{g615_eq1}
	(\forall\xi\in\gg^*\setminus[\gg,\gg]^\perp)\quad 
	\gg(\xi)=\zg\dotplus\RR Y(\xi\vert_{[\gg,\gg]})
	=\spa\{X_1,X_2,X_3,Y(\xi\vert_{[\gg,\gg]})\}, 
	\end{equation} 
	This equality gives the subalgebra denoted by $\n$ in \cite[Rem. 22]{Ou19}, which is thus the isotropy subalgebra rather than a polarization. 
	In fact, as it also follows from the reasoning below, the polarizations of linear functionals on the nilpotent Lie algebra $\gg_{6,15}$ are either 5-dimensional or 6-dimensional. 
	
	We compute the Vergne polarization mapping 
	$$(\forall\xi\in\gg^*)\quad \pg_{\alg}(\xi)=\gg_1(\xi\vert_{\gg_1})+\gg_2(\xi\vert_{\gg_2})+
	\gg_3(\xi\vert_{\gg_3})+\gg_4(\xi\vert_{\gg_4})+\gg_5(\xi\vert_{\gg_5})
	+\gg(\xi).$$
	For $j=1,2,3,4$, the Lie algebra $\gg_j$ is abelian 
	hence $\gg_j(\xi\vert_{\gg_j})=\gg_j$ and 
	then 
	\begin{equation}
	\label{g615_eq2}
	(\forall\xi\in\gg^*)\quad \pg_{\alg}(\xi)=\gg_4+\gg_5(\xi\vert_{\gg_5})+\gg(\xi).
	\end{equation}
	Here $\gg_5=\spa\{X_1,X_3\}\dotplus\spa\{X_2,X_4,X_5\}$ as the direct sum of a 2-dimensional Lie algebra and a 3-dimensional Heisenberg algebra with its center $\RR X_2$, and it then easily follows that 
	\begin{equation*}
	\gg_5(\xi\vert_{\gg_5})
	=\begin{cases}
	\spa\{X_1,X_2,X_3\}&\text{ if }\xi_2\ne0,\\
	\spa\{X_1,X_2,X_3,X_4,X_5\}&\text{ if }\xi_2=0,
	\end{cases}
	=\begin{cases}
	\gg_3&\text{ if }\xi_2\ne0,\\
	\gg_5&\text{ if }\xi_2=0. 
	\end{cases}
	\end{equation*}
	Then, using \eqref{g615_eq1} and \eqref{g615_eq2}, we easily obtain 
	\begin{equation}
	\label{g615_eq3}
	\pg_{\alg}(\xi)
	=\begin{cases}
	\spa\{X_1,X_2,X_3,X_4,-\xi_1X_5+\xi_2X_6\}&\text{ if }\xi_2\ne0,\\
	\spa\{X_1,X_2,X_3,X_4,X_5\}&\text{ if }\xi_2=0. 
	\end{cases}
	\end{equation}
	On the other hand, for $\xi\in\gg^*$, we have by the above computation  $\dim(\gg_j(\xi\vert_{\gg_j}))=j$ if $j=1,2,3,4$, 
	while 
	$$\dim\gg_5(\xi\vert_{\gg_5})=
	\begin{cases}
	3&\text{ if }\xi_2\ne 0,\\
	5&\text{ if }\xi_2=0,
	\end{cases}$$
	and 
	$$\dim\gg(\xi)=
	\begin{cases}
	4&\text{ if }\xi\in\gg^*\setminus\gg_3^\perp,\text{ i.e., }(\xi_1,\xi_2,\xi_3)\in\RR^3\setminus\{(0,0,0)\},\\
	6&\text{ if }\xi\in\gg_3^\perp,\text{ i.e., } \xi_1=\xi_2=\xi_3=0.
	\end{cases}$$
	Let us now find explicit form for the
	subsets 
	\begin{equation*}
	\Xi_{(k_1,k_2,k_3,k_4,k_5,k_6)}=\{\xi\in\gg^*\mid \dim(\gg_j(\xi\vert_{\gg_j}))=k_j\text{ for }j=1,2,3,4,5,6\}
	\end{equation*}
	where $1\le k_j\le j$ for $j=1,2,3,4,5$. 
	The above computation shows that the set  $\Xi_{(k_1,k_2,k_3,k_4,k_5,k_6)}$ is non-empty 
	iff $k_j=j$ for $j=1,2,3,4$ and 
	$(k_5,k_6)\in\{(3,4),(5,4),(5,6)\}$. 
	More specifically, 
	\begin{align*}
	\Xi_{(1,2,3,4,3,4)}&=\{\xi\in\gg^*\mid \xi_2\ne 0\},\\
	\Xi_{(1,2,3,4,5,4)}&=\{\xi\in\gg^*\mid \xi_2=0\text{ and }(\xi_1,\xi_3)\in\RR^2\setminus\{(0,0)\}\},\\
	\Xi_{(1,2,3,4,5,6)}&=\{\xi\in\gg^*\mid \xi_1=\xi_2=\xi_3=0\}.
	\end{align*} 
	By Theorem~\ref{polcont}, the Vergne polarization mapping $\pg_{\alg}\colon\gg^*\to\Gr_{\alg}(\gg)$ is con\-ti\-nuous on each of the subsets
	$\Xi_{(k_1,k_2,k_3,k_4,k_5,k_6)}$. 
	This can also be directly checked, using \eqref{g615_eq3}.
	In particular, $\pg_{\alg}$ is continuous on the open dense subset $\Xi_{(1,2,3,4,3,4)}\subseteq\gg^*$.
	
	We note that $\Xi_{(1,2,3,2,3)}$ is not a maximal domain on which $\pg_{\alg}$ is continuous. 
	For instance, it easily follows by \eqref{g615_eq3} 
	that $\pg_{\alg}$ is continuous on the larger open set 
	$\{\xi\in\gg^*\mid(\xi_1,\xi_2)\in\RR^2\setminus\{(0,0)\}\}$. 
	However, $\pg_{\alg}$ is not continuous on the whole open set  $\gg^*\setminus[\gg,\gg]^\perp=\{\xi\in\gg^*\mid (\xi_1,\xi_2,\xi_3)\in\RR^3\setminus\{(0,0,0)\}\}$. 
	For instance,  it directly follows again by \eqref{g615_eq3}  that the mapping $\pg_{\alg}$ is not continuous on the subset 
	$\{\xi\in\gg^*\mid \xi_1=0,\ (\xi_2,\xi_3)\in\RR^2\setminus\{(0,0)\}\}$. 
\end{example}

In the above example, we now describe the topology of the space of 2-dimensional coadjoint orbits in $\gg^*$, which in particular shows that this topological space is Hausdorff and is homeomorphic to $S^2\times\RR^2$, 
where $S^2$ is the unit sphere in the space~$\RR^3$. 
This topological space was earlier studied by other methods, for instance in \cite[Ex. 6.3.5]{Ec96} and \cite[\S 2]{ArSoKaSc99}. 

\begin{proposition}
	\label{g615_H}
	With the notation of Example~\ref{g615}, the following assertions hold: 
	\begin{enumerate}[{\rm(i)}]
		\item\label{g615_H_item1}
		The polynomial function $C\colon\gg^*\to\RR$, $C(\xi):=\xi_2\xi_6+\xi_3\xi_4-\xi_1\xi_5$ is constant on the coadjoint orbits in~$\gg^*$. 
		\item\label{g615_H_item2}
		The mapping 
		\begin{equation*}
		\Psi\colon (\gg^*\setminus\zg^\perp)/G\to(\zg^*\setminus\{0\})\times\RR, \quad 
		\Psi(G\xi)=(\xi\vert_\zg,C(\xi))
		\end{equation*}
		is well defined and is a homeomorphism. 
	\end{enumerate}
\end{proposition}

\begin{proof}
	\eqref{g54_H_item1} 
	We recall from \cite{Ni83} that, 
	if we identify $\gg^*$ to $\RR^6$ via the mapping $\xi\mapsto(\xi_j)_{1\le j\le 6}$, then we have the following description of the coadjoint orbits $\Oc_\xi:=G\xi\subseteq\gg^*\simeq\RR^6$ for $\xi\in \gg^*$: 
	\begin{itemize}
		\item If $\xi_2\ne0$, then 
		$$\Oc_\xi=\{(\xi_1,\xi_2,\xi_3,y_4,y_5,\frac{1}{\xi_2}(C(\xi)+\xi_1y_5-\xi_3y_4))\in\RR^6\mid y_4,y_5\in\RR\}$$
		hence $C\vert_{\Oc_\xi}$ is constant by a direct verification. 
		\item If $\xi_2=0\ne\xi_1$ then 
		$$\Oc_\xi=\{(\xi_1,0,\xi_3,y_4,\frac{1}{\xi_1}(-C(\xi)+\xi_3y_4),y_6)\in\RR^6\mid y_4,y_6\in\RR\}$$
		hence $C\vert_{\Oc_\xi}$ is constant again by a direct verification. 
		\item If $\xi_1=\xi_2=0\ne\xi_3$ then 
		$$\Oc_\xi=\{(0,0,\xi_3,\frac{1}{\xi_3}C(\xi),y_5,y_6)\in\RR^6\mid y_5,y_6\in\RR\}$$
		hence $C\vert_{\Oc_\xi}$ is constant again by a direct verification. 
		\item If $\xi_1=\xi_2=\xi_3=0$ then 
		$$\Oc_\xi=\{(0,0,0,\xi_4,\xi_5,\xi_6)\}=\{\xi\}$$
		hence $C\vert_{\Oc_\xi}$ is clearly constant. 
	\end{itemize}
	
	\eqref{g54_H_item2} 
	It follows by Assertion~\eqref{g54_H_item1} that the mapping $\Psi$ is well defined. 
	Moreover, the specific formulas in the proof of  Assertion~\eqref{g54_H_item1} show that if $(\xi_1,\xi_2,\xi_3)\in\RR^3\setminus\{(0,0,0)\}$, then the coadjoint orbit $\Oc_\xi$ is uniquely determined by $\xi_1,\xi_2,\xi_3,C(\xi)$. 
	Since $\zg=\spa\{X_1,X_2,X_3\}$, this shows that the mapping $\Psi$ is injective. 
	To see that $\Psi$ is surjective, it suffices to check that for every $(\xi_1,\xi_2,\xi_3)\in\RR^3\setminus\{(0,0,0)\}$ the function $\varphi_{\xi_1,\xi_2,x_3}\colon \RR^3\to\RR$, $\varphi_{\xi_1,\xi_2,x_3}(\xi_4,\xi_5,\xi_6):=C(\xi_1,\xi_2.\xi_3,\xi_4,\xi_5,x_6)$, is surjective, which is straightforward since  $\varphi_{\xi_1,\xi_2,x_3}$ is a linear function that does not vanish identically. 
	
	We now need the quotient map $q\colon\gg^*\setminus\zg^\perp\to(\gg^*\setminus\zg^\perp)/G$, $q(\xi):=G\xi=\Oc_\xi$. 
	It is well known that this mapping $q$ is continuous and open. 
	On the other hand, it is clear that the composition $\Psi\circ q\colon \gg^*\setminus\zg^\perp\to(\zg^*\setminus\{0\})\times\RR$ is a smooth function, and in particular continuous, hence $\Psi$ is in turn continuous. 
	
	To complete the proof of the fact that $\Psi$ is a homeomorphism, we use again the fact that $q$ is an open mapping, hence it suffices to show that $\Psi\circ q$ is an open mapping. 
	To this end we check that the smooth mapping $\Psi\circ q$ is a submersion. 
	In fact, since $(\Psi\circ q)(\xi_1,\xi_2,\xi_3,\xi_4,\xi_5,x_6)=(\xi_1,\xi_2,x_3,C(\xi))$, it follows that 
	the differential of $\Psi\circ q$ at an arbitrary point 
	$(\xi_1,\xi_2,\xi_3,\xi_4,\xi_5,x_6)\in\RR^6$ with $(\xi_1,\xi_2,x_3)\ne(0,0,0)$ is given by the matrix 
	$$\begin{pmatrix}
	1 &  0 & 0 & 0 & 0 & 0\\
	0 & 1 & 0 & 0 & 0 & 0\\
	0 & 0 & 1 & 0 & 0 & 0\\
	-\xi_5 & \xi_6 & \xi_4 & \xi_3 & -\xi_1 & \xi_2
	\end{pmatrix}
	$$
	whose rank is clearly equal to~4 since $(\xi_1,\xi_2,\xi_3)\ne(0,0,0)$. 
	Therefore the mapping $\Psi\circ q\colon \gg^*\setminus\zg^\perp\to(\zg^*\setminus\{0\})\times\RR$  is a submersion, 
	and this completes the proof. 
\end{proof}

\subsection*{The nilpotent Lie algebra $\gg_{5,4}$} 

\begin{example}
\label{g54}
\normalfont
We consider the nilpotent real Lie algebra $\gg$ denoted by $\gg_{5,4}$ in \cite{Dix58} and \cite{Ni83}, defined by a basis $X_1,X_2,X_3,X_4,X_5$ satisfying the commutation relations 
$$[X_5,X_4]=X_3,\ [X_5,X_3]=X_2,\ [X_4,X_3]=X_1.$$
We consider the Jordan-H\"older sequence 
$$\gg_0=\{0\}\subsetneqq\gg_1\subsetneqq\gg_2\subsetneqq\gg_3\subsetneqq\gg_4\subsetneqq\gg_5=\gg$$
defined by $\gg_j=\spa\{X_i\mid 1\le i\le j\}$ for $j=1,\dots,5$. 
Then the center of $\gg$ is $\zg:=\gg_2$, and moreover $[\gg,\gg]=\gg_3$. 
For every $\xi\in\gg^*$ we denote $\xi_j:=\langle\xi,X_j\rangle$ for $j=1,\dots,5$. 
Similarly, for $\eta\in[\gg,\gg^*]=\gg_3^*$ we denote $\eta_j:=\langle\eta,X_j\rangle$ for $j=1,2,3$. 
With this notation we define the mapping 
$$Y\colon[\gg,\gg]^*\to\gg,\quad Y(\eta):=\eta_3X_3-\eta_2X_4+\eta_1X_5$$
and then we have, cf. \cite{Ni83}, 
\begin{equation}
\label{g54_eq1}
(\forall\xi\in\gg^*\setminus[\gg,\gg]^\perp)\quad 
\gg(\xi)=\zg\dotplus\RR Y(\xi\vert_{[\gg,\gg]})
=\spa\{X_1,X_2,Y(\xi\vert_{[\gg,\gg]})\}, 
\end{equation} 
Now let us consider the Vergne polarization mapping 
$$(\forall\xi\in\gg^*)\quad \pg_{\alg}(\xi)=\gg_1(\xi\vert_{\gg_1})+\gg_2(\xi\vert_{\gg_2})+
\gg_3(\xi\vert_{\gg_3})+\gg_4(\xi\vert_{\gg_4})+\gg_5(\xi).$$
For $j=1,2,3$, the Lie algebra $\gg_j$ is abelian 
hence $\gg_j(\xi\vert_{\gg_j})=\gg_j$ and 
then 
\begin{equation}
\label{g54_eq2}
(\forall\xi\in\gg^*)\quad \pg_{\alg}(\xi)=\gg_3+\gg_4(\xi\vert_{\gg_4})+\gg(\xi).
\end{equation}
Here $\gg_4=\RR X_1\dotplus\spa\{X_1,X_3,X_4\}$ as the direct sum of a 1-dimensional Lie algebra and a 3-dimensional Heisenberg algebra with its center $\RR X_1$, and it then easily follows that 
\begin{equation*}
\gg_4(\xi\vert_{\gg_4})
=\begin{cases}
\spa\{X_1,X_2\}&\text{ if }\xi_1\ne0,\\
\spa\{X_1,X_2,X_3,X_4\}&\text{ if }\xi_1=0,
\end{cases}
=\begin{cases}
\gg_2&\text{ if }\xi_1\ne0,\\
\gg_4&\text{ if }\xi_1=0. 
\end{cases}
\end{equation*}
Then, using \eqref{g54_eq1} and \eqref{g54_eq2}, we easily obtain 
\begin{equation}
\label{g54_eq3}
\pg_{\alg}(\xi)
=\begin{cases}
\spa\{X_1,X_2,X_3,-\xi_2X_4+\xi_1X_5\}&\text{ if }\xi_1\ne0,\\
\spa\{X_1,X_2,X_3,X_4\}&\text{ if }\xi_1=0. 
\end{cases}
\end{equation}
On the other hand, for $\xi\in\gg^*$ we have by the above computation  $\dim(\gg_j(\xi\vert_{\gg_j}))=j$ if $j=1,2,3$, 
while 
$$\dim\gg_4(\xi\vert_{\gg_4})=
\begin{cases}
2&\text{ if }\xi_1\ne 0,\\
3&\text{ if }\xi_1=0,
\end{cases}$$
and 
$$\dim\gg(\xi)=
\begin{cases}
3&\text{ if }\xi\in\gg^*\setminus\gg_3^\perp,\text{ i.e., }(\xi_1,\xi_2,\xi_3)\in\RR^3\setminus\{(0,0,0)\},\\
5&\text{ if }\xi\in\gg_3^\perp,\text{ i.e., } \xi_1=\xi_2=\xi_3=0.
\end{cases}$$
Let us now find explicit form for the subsets
\begin{equation*}
\Xi_{(k_1,k_2,k_3,k_4,k_5)}=\{\xi\in\gg^*\mid \dim(\gg_j(\xi\vert_{\gg_j}))=k_j\text{ for }j=1,2,3,4,5\},
\end{equation*}
where $1\le k_j\le j$ for $j=1,2,3,4,5$. 
The above computation shows that the set $\Xi_{(k_1,k_2,k_3,k_4,k_5)}$  is non-empty
if and only if $k_j=j$ for $j=1,2,3$ and 
 $(k_4,k_5)\in\{(2,3),(4,3),(4,5)\}$. 
More specifically, 
\begin{align*}
\Xi_{(1,2,3,2,3)}&=\{\xi\in\gg^*\mid \xi_1\ne 0\},\\
\Xi_{(1,2,3,4,3)}&=\{\xi\in\gg^*\mid \xi_1=0\text{ and }(\xi_2,\xi_3)\in\RR^2\setminus\{(0,0)\}\},\\
\Xi_{(1,2,3,4,5)}&=\{\xi\in\gg^*\mid \xi_1=\xi_2=\xi_3=0\}.
\end{align*}
By Theorem~\ref{polcont}, the Vergne polarization mapping $\pg_{\alg}\colon\gg^*\to\Gr_{\alg}(\gg)$ is continuous on each of the subsets  $\Xi_{(k_1,k_2,k_3,k_4,k_5)}$. 
This can also be directly checked, using \eqref{g54_eq3}. 
In particular,  $\pg_{\alg}$ is continuous on the open dense subset $\Xi_{(1,2,3,2,3)}\subseteq\gg^*$. 

We note that $\Xi_{(1,2,3,2,3)}$ is not a maximal domain on which $\pg_{\alg}$ is continuous. 
For instance, it easily follows by \eqref{g54_eq3} 
that $\pg_{\alg}$ is continuous on the larger open set 
$\{\xi\in\gg^*\mid(\xi_1,\xi_2)\in\RR^2\setminus\{(0,0)\}\}$. 
However, $\pg_{\alg}$ is not continuous on the whole open set  $\gg^*\setminus[\gg,\gg]^\perp=\{\xi\in\gg^*\mid (\xi_1,\xi_2,\xi_3)\in\RR^3\setminus\{(0,0,0)\}\}$. 
For instance,  it directly follows again by \eqref{g54_eq3}  that the mapping $\pg_{\alg}$ is not continuous on the subset 
$\{\xi\in\gg^*\mid \xi_2=0,\ (\xi_1,\xi_3)\in\RR^2\setminus\{(0,0)\}\}$. 
\end{example}

In the above example, we now describe the topology of the space of 2-dimensional coadjoint orbits in $\gg^*$

\begin{proposition}
\label{g54_H}
With the notation of Example~\ref{g54}, the following assertions hold: 
\begin{enumerate}[{\rm(i)}]
	\item\label{g54_H_item1}
	The polynomial function $C\colon\gg^*\to\RR$, $C(\xi):=2\xi_1\xi_5-2\xi_2\xi_4+\xi_3^2$ is constant on the coadjoint orbits in~$\gg^*$. 
	\item\label{g54_H_item2}
	The mapping 
	\begin{equation*}
	\Psi\colon (\gg^*\setminus[\gg,\gg]^\perp)/G\to(\zg^*\setminus\{0\})\times\RR, \quad 
	\Psi(G\xi)=(\xi\vert_\zg,C(\xi))
	\end{equation*}
	is well defined, continuous, open, and surjective.  
	Moreover, for any coadjoint orbit $\Oc\in(\gg^*\setminus[\gg,\gg]^\perp)/G$ 
	we have $\{\Oc\}\neq\Psi^{-1}(\Psi(\Oc))$ if and only if $\Oc\subseteq\zg^\perp$. 
	If this is the case, then $\Psi^{-1}(\Psi(\Oc))=\{\Oc,-\Oc\}$. 
\end{enumerate}
\end{proposition}

\begin{proof}
\eqref{g54_H_item1} 
The polynomial function $C\colon\gg^*\to\RR$ corresponds (via the symmetrization map) to an element in the center of the universal enveloping algebra of $\gg$ by \cite[Prop. 2]{Dix58}, 
It then follows by \cite[Cor. 3.3.3(c)]{CG90} that the function $C$ is constant on the coadjoint orbts in~$\gg^*$. 

This fact can also be obtained by a more concrete method, 
which we now indicate since it will be needed in the proof of Assertion~\eqref{g54_H_item2} below. 
Namely, we recall from \cite{Ni83} that, 
if we identify $\gg^*$ to $\RR^5$ via the mapping $\xi\mapsto(\xi_j)_{1\le j\le 5}$, then we have the following description of the coadjoint orbits $\Oc_\xi:=G\xi\subseteq\gg^*\simeq\RR^5$ for $\xi\in \gg^*$: 
\begin{itemize}
	\item If $\xi_1\ne0$, then 
	$$\Oc_\xi=\{(\xi_1,\xi_2,y_3,y_4,\frac{1}{2\xi_1}(C(\xi)+2\xi_2y_4-y_3^2))\in\RR^5\mid y_3,y_4\in\RR\}$$
	hence $C\vert_{\Oc_\xi}$ is constant by a direct verification. 
	\item If $\xi_1=0\ne\xi_2$ then 
	$$\Oc_\xi=\{(0,\xi_2,y_3,\frac{1}{2\xi_2}(-C(\xi)+y_3^2),y_5)\in\RR^5\mid y_3,y_5\in\RR\}$$
	hence $C\vert_{\Oc_\xi}$ is constant again by a direct verification. 
	\item If $\xi_1=\xi_2=0\ne\xi_3$ then 
	$$\Oc_\xi=\{(0,0,\xi_3,y_4,y_5)\in\RR^5\mid y_4,y_5\in\RR\}$$
	hence $C\vert_{\Oc_\xi}$ is constant again by a direct verification. 
	\item If $\xi_1=\xi_2=\xi_3=0$ then 
	$$\Oc_\xi=\{(0,0,0,\xi_4,\xi_5)\}=\{\xi\}$$
	hence $C\vert_{\Oc_\xi}$ is clearly constant. 
\end{itemize}

\eqref{g54_H_item2} 
By Assertion~\eqref{g54_H_item1}, the mapping $\Psi$ is well defined. 
Moreover, the specific formulas in the second proof of  Assertion~\eqref{g54_H_item1} show that if $(\xi_1,\xi_2)\in\RR^2\setminus\{(0,0)\}$, then the coadjoint orbit $\Oc_\xi$ is uniquely determined by the values $\xi_1,\xi_2,C(\xi)$, hence by $\Psi(\Oc_\xi)$. 
Moreover, the condition $(\xi_1,\xi_2)\in\RR^2\setminus\{(0,0)\}$ is equivalent to $\Oc_\xi\not\subseteq\zg^\perp$. 
On the other hand, if $\xi_1=\xi_2=0$, then $\Oc_\xi=\{(0,0,\xi_3,y_4,y_5)\in\RR^5\mid y_4,y_5\in\RR\}$, 
while $C(\xi)=\xi_3^2$. 
Thus, since  $\zg=\spa\{X_1,X_2\}$, if $\xi,\zeta\in\zg^\perp\setminus[\gg,\gg]^\perp$ and $\Oc_\zeta\ne\Oc_\xi$, 
then we have 
$$\Oc_\zeta=-\Oc_\xi\iff\zeta_3=-\xi_3\iff C(\zeta)=C(\xi)\iff \Psi(\Oc_\zeta)=\Psi(\Oc_\xi).$$
This exactly describes to which extent the mapping $\Psi$ fails to be injective. 
 
To see that $\Psi$ is surjective, it suffices to check that for every $(\xi_1,\xi_2)\in\RR^2\setminus\{(0,0)\}$ the function $\varphi_{\xi_1,\xi_2}\colon \RR^3\to\RR$, $\varphi_{\xi_1,\xi_2}(\xi_3,\xi_4,\xi_5):=C(\xi_1,\xi_2,\xi_3,\xi_4,\xi_5)$, is surjective. 
In fact, we have even $\varphi_{\xi_1,\xi_2}(\{0\}\times\RR^2)=\RR$ by the definition of the function~$C$. 

We now need the quotient map $q\colon\gg^*\setminus[\gg,\gg]^\perp\to(\gg^*\setminus[\gg,\gg]^\perp)/G$, $q(\xi):=G\xi=\Oc_\xi$, which is
continuous and open. 
On the other hand, it is clear that the composition $\Psi\circ q\colon \gg^*\setminus[\gg,\gg]^\perp\to(\zg^*\setminus\{0\})\times\RR$ is a smooth function, and in particular continuous, hence $\Psi$ is in turn continuous. 

To complete the proof of the fact that $\Psi$ is a homeomorphism, we use again the fact that $q$ is an open mapping, hence it suffices to show that $\Psi\circ q$ is an open mapping. 
To this end we check that the smooth mapping $\Psi\circ q$ is a submersion. 
In fact, since $(\Psi\circ q)(\xi_1,\xi_2,\xi_3,\xi_4,\xi_5)=(\xi_1,\xi_2,C(\xi))$, it follows that 
the differential of $\Psi\circ q$ at an arbitrary point 
$(\xi_1,\xi_2,\xi_3,\xi_4,\xi_5)\in\RR^5$ with $(\xi_1,\xi_2)\ne(0,0)$ is given by the matrix 
$$\begin{pmatrix}
1 &  0 & 0 & 0 & 0 \\
0 & 1 & 0 & 0 & 0 \\
2\xi_5 & -2\xi_4 & 2\xi_3 & -2\xi_2 & 2\xi_1
\end{pmatrix}
$$
whose rank is clearly equal to~3 since $(\xi_1,\xi_2)\ne(0,0)$. 
This shows that the mapping $\Psi\circ q\colon \gg^*\setminus[\gg,\gg]^\perp\to(\zg^*\setminus\{0\})\times\RR$  is a submersion, 
and the proof is complete. 
\end{proof}

\begin{remark}
\normalfont 
Proposition~\ref{g54_H} in particular shows that the quotient topological space $(\gg^*\setminus[\gg,\gg]^\perp)/G$ is not Hausdorff and its complete regularization is homeomorphic to $S^1\times\RR^2$, 
where $S^1$ is the unit circle in the plane~$\RR^2$. 
This sheds extra light on the topology of the primitive ideal space of the nilpotent Lie group $G_{5,4}$ discussed in \cite[Ex. 1]{ArKa97}. 
In particular  it shows that if we denote by $\Jc$ the closed two-sided ideal of $C^*(G)$ whose primitive ideal space corresponds to the set of 2-dimensional coadjoint orbits $(\gg^*\setminus[\gg,\gg]^\perp)/G$, then $\Jc$ is a quasi-standard $C^*$-algebra in the sense of \cite{ArSo90}. 
\end{remark}

\subsection*{Acknowledgment}
The research of the second-named author was supported by a grant of the  Ministry of Research, Innovation and Digitization, CNCS/CCCDI -- UEFISCDI, project number PN-III-P4-ID-PCE-2020-0878, within PNCDI III.

\end{document}